\documentclass[11pt]{article}
\usepackage[noadjust]{cite}
\usepackage{enumitem}

\usepackage{etex}
\usepackage[all]{xy}
\usepackage{cite}
\usepackage{amsfonts}
\usepackage{amsthm} 
\usepackage{graphicx}
\usepackage{mathrsfs}
\usepackage{bm}
\usepackage{amssymb,amsmath} % font used for R in Real numbers
\usepackage{makecell}
\usepackage{tikz}
\usepackage{pgf}
\usepackage{tikz}
\usetikzlibrary{patterns}
\usepackage{pgffor}
\usepackage{pgfcalendar}
\usepackage{pgfpages}
\usepackage{shuffle,yfonts}
\usepackage{mathtools}
 \allowdisplaybreaks

\DeclareFontFamily{U}{shuffle}{}
\DeclareFontShape{U}{shuffle}{m}{n}{ <-8>shuffle7 <8->shuffle10}{}

\newcommand{\nc}{\newcommand}
\nc{\AMZV}{\mathsf {AMZV}}
\nc{\ud}{\mathrm{d}}
\nc{\ES}{\mathsf {ES}}
\nc{\MZV}{\mathsf {MZV}}
\nc{\MtV}{\mathsf {MtV}}
\nc{\MTV}{\mathsf {MTV}}
\nc{\MSV}{\mathsf {MSV}}
\nc{\MMV}{\mathsf {MMV}}
\nc{\MMVo}{\mathsf {MMVo}}
\nc{\MMVe}{\mathsf {MMVe}}
\nc{\AMMV}{\mathsf {AMMV}}
\nc{\AMTV}{\mathsf {AMTV}}
\nc{\AMtV}{\mathsf {AMtV}}
\nc{\AMSV}{\mathsf {AMSV}}
\nc{\CMZV}{\mathsf {CMZV}}
\nc{\sha}{\shuffle}
\nc{\cst}{\rotatebox[origin=c]{180}{$\sha$}}
\nc{\cstt}{\rotatebox[origin=c]{180}{$\scriptstyle \sha$}}

\nc{\de}{\delta}
\nc{\DD}{{\mathbb D}}
\nc{\anbb}[1]{\left\langle#1\right\rangle}
\nc{\bibb}[1]{\left\{#1\right\}}
\nc{\mibb}[1]{\left[#1\right]}
\nc{\smbb}[1]{\left(#1\right)}
\nc{\doubb}[1]{\llbracket#1\rrbracket}
\nc{\dm}[1]{\left|#1\right|}

\nc{\Gbinom}[2]{\genfrac{(}{)}{0mm}{0}{#1}{#2}}
\nc{\gbinom}[2]{\genfrac{(}{)}{0mm}{1}{#1}{#2}}
\nc{\Rbinom}[2]{\genfrac{\langle}{\rangle}{0mm}{0}{#1}{#2}}
\nc{\rbinom}[2]{\genfrac{\langle}{\rangle}{0mm}{1}{#1}{#2}}
\nc{\Qbinom}[2]{\genfrac{[}{]}{0mm}{0}{#1}{#2}_q}
\nc{\qbinom}[2]{\genfrac{[}{]}{0mm}{1}{#1}{#2}_q}

\nc{\binq}[2]{\genfrac{[}{]}{0mm}{0}{#1}{#2}}
\nc{\tbnq}[2]{\genfrac{[}{]}{0mm}{1}{#1}{#2}}
\nc{\cinq}[2]{\genfrac{\{}{\}}{0mm}{0}{#1}{#2}}
\nc{\tcnq}[2]{\genfrac{\{}{\}}{0mm}{1}{#1}{#2}}

\nc{\mfrac}[2]{\genfrac{}{}{0pt}{}{#1}{#2}}
\nc{\tf}{\tfrac}
\nc{\db}{{\mathbb D}}
\nc{\bara}{{\bar{a}}}
\nc{\barb}{{\bar{b}}}
\nc{\ta}{{\tilde{a}}}
\nc{\dk}{{\mathbb K}}
\nc{\ola}{\overleftarrow}
\nc{\ora}{\overrightarrow}
\nc{\lra}{\longrightarrow}
\nc{\Lra}{\Longrightarrow}
\nc\Res{{\rm Res}}
\nc\setX{{\mathsf{X}}}
\nc\fA{{\mathfrak{A}}}
\nc\evaM{{\texttt{M}}}
\nc\evaML{{\text{\em{\texttt{M}}}}}
\nc\z{{\texttt{z}}}
\nc\emz{\emph{\texttt{z}}}
\nc\tx{{\texttt{x}}}
\nc\txp{{\tx_1}} % textstyle x positive 1
\nc\txn{{\tx_{-1}}} % textstyle x negative 1
\nc\neo{{1}}
\nc{\yi}{{1}}
\nc\one{{-1}}
\nc\gD{{\Delta}}
\nc\eps{{\varepsilon}}
\nc{\bfe}{{\boldsymbol{\sl{e}}}}
\nc{\bfi}{{\boldsymbol{\sl{i}}}}
\nc{\bfj}{{\boldsymbol{\sl{j}}}}
\nc{\bfk}{{\boldsymbol{\sl{k}}}}
\nc{\bfl}{{\boldsymbol{\sl{l}}}}
\nc{\bfm}{{\boldsymbol{\sl{m}}}}
\nc{\bfn}{{\boldsymbol{\sl{n}}}}
\nc{\bfs}{{\boldsymbol{\sl{s}}}}
\nc{\bft}{{\boldsymbol{\sl{t}}}}
\nc{\bfx}{{\boldsymbol{\sl{x}}}}
\nc{\bfz}{{\boldsymbol{\sl{z}}}}
\nc{\bfq}{{\boldsymbol{\sl{q}}}}
\nc{\bfp}{{\boldsymbol{\sl{p}}}}
\nc\bfgs{{\boldsymbol \gs}}
\nc\bfgl{{\boldsymbol \lambda}}
\nc\bfsi{{\boldsymbol \gs}}
\nc\bfet{{\boldsymbol \eta}}
\nc\bfeta{{\boldsymbol \eta}}
\nc\bfeps{{\boldsymbol \eps}}
\nc\bfga{{\boldsymbol \ga}}
\nc\bfgb{{\boldsymbol \gb}}
\nc\mmu{{\boldsymbol \mu}}
\nc\bfone{{\bf 1}}
\nc{\myone}{{1}}

 \nc{\calA}{{\mathcal A}}
 \nc{\calB}{{\mathcal B}}
 \nc{\calC}{{\mathcal C}}
 \nc{\calD}{{\mathcal D}}
 \nc{\calE}{{\mathcal E}}
 \nc{\calF}{{\mathcal F}}
 \nc{\calG}{{\mathcal G}}
 \nc{\calH}{{\mathcal H}}
 \nc{\calI}{{\mathcal I}}
 \nc{\calJ}{{\mathcal J}}
 \nc{\calK}{{\mathcal K}}
 \nc{\calL}{{\mathcal L}}
 \nc{\calM}{{\mathcal M}}
 \nc{\calN}{{\mathcal N}}
 \nc{\calO}{{\mathcal O}}
 \nc{\calP}{{\mathcal P}}
 \nc{\calQ}{{\mathcal Q}}
 \nc{\calR}{{\mathcal R}}
 \nc{\calS}{{\mathcal S}}
 \nc{\calT}{{\mathcal T}}
 \nc{\calU}{{\mathcal U}}
 \nc{\calV}{{\mathcal V}}
 \nc{\calW}{{\mathcal W}}
 \nc{\calX}{{\mathcal X}}
 \nc{\calY}{{\mathcal Y}}
 \nc{\calZ}{{\mathcal Z}}
 \nc{\cala}{{\mathcal a}}
 \nc{\calb}{{\mathcal b}}
 \nc{\calc}{{\mathcal c}}
 \nc{\cald}{{\mathcal d}}
 \nc{\cale}{{\mathcal e}}
 \nc{\calf}{{\mathcal f}}
 \nc{\calg}{{\mathcal g}}
 \nc{\calh}{{\mathcal h}}
 \nc{\cali}{{\mathcal i}}
 \nc{\calj}{{\mathcal j}}
 \nc{\calk}{{\mathcal k}}
 \nc{\call}{{\mathcal l}}
 \nc{\calm}{{\mathcal m}}
 \nc{\caln}{{\mathcal n}}
 \nc{\calo}{{\mathcal o}}
 \nc{\calp}{{\mathsf p}}
 \nc{\calq}{{\mathcal q}}
 \nc{\calr}{{\mathcal r}}
 \nc{\cals}{{\mathcal s}}
 \nc{\calt}{{\mathcal t}}
 \nc{\calu}{{\mathcal u}}
 \nc{\calv}{{\mathcal v}}
 \nc{\calw}{{\mathcal w}}
 \nc{\calx}{{\mathcal x}}
 \nc{\caly}{{\mathcal y}}
 \nc{\calz}{{\mathcal z}}
 \nc{\ot}{{\otimes}}
\usetikzlibrary{arrows,shapes,chains}

\catcode`!=11
\let\!int\int \def\int{\displaystyle\!int}
\let\!lim\lim \def\lim{\displaystyle\!lim}
\let\!sum\sum \def\sum{\displaystyle\!sum}
\let\!sup\sup \def\sup{\displaystyle\!sup}
\let\!inf\inf \def\inf{\displaystyle\!inf}
\let\!cap\cap \def\cap{\displaystyle\!cap}
\let\!max\max \def\max{\displaystyle\!max}
\let\!min\min \def\min{\displaystyle\!min}
\let\!frac\frac \def\frac{\displaystyle\!frac}
\catcode`!=12

\nc{\gam}{{\gamma}}
\nc{\gG}{{\Gamma}}
\nc{\om}{{\omega}}
\nc{\vep}{{\varepsilon}}
\nc{\ga}{{\alpha}}
\nc{\gd}{{\delta}}
\nc{\gl}{{\lambda}}
\nc{\gb}{{\beta}}
\nc{\gf}{{\varphi}}
\nc{\gs}{{\sigma}}
\nc{\gk}{{\kappa}}
\nc{\gS}{\Sigma}
\let\oldsection\section
\renewcommand\section{\setcounter{equation}{0}\oldsection}

\allowdisplaybreaks

\DeclareMathOperator*{\dep}{dep}

\DeclareMathOperator{\dch}{\texttt{dch}}

\nc\UU{\mbox{\bfseries U}}
\nc\FF{\mbox{\bfseries \itshape F}}
\nc\h{\mbox{\bfseries \itshape h}}\nc\dd{\mbox{d}}
\nc\g{\mbox{\bfseries \itshape g}}
\nc\xx{\mbox{\bfseries \itshape x}}
\nc\tz{\tilde\zeta}

\def\PP{\mathbb{P}}

\def\N{\mathbb{N}}
\def\Z{\mathbb{Z}}
\def\Q{\mathbb{Q}}

\def\xx{\left(\frac{1-x}{1+x} \right)}

\def\ol{\overline}
\nc\divg{{\text{div}}}
\theoremstyle{plain}
\newtheorem{thm}{Theorem}[section]
\newtheorem{lem}[thm]{Lemma}

\newtheorem{cor}[thm]{Corollary}
\newtheorem{conj}[thm]{Conjecture}
\newtheorem{prop}[thm]{Proposition}
\newtheorem{prob}[thm]{Problem}
\theoremstyle{definition}

\newtheorem{rem}[thm]{Remark}

\nc{\cicc}[1]{{}_{{}^{ \bigcirc\hskip-1.2ex{#1}\hskip.3ex{}}}}
\nc{\cic}[1]{{}^{\bigcirc\hskip-1.15ex{\raisebox{-0.015cm}{\text{$\scriptscriptstyle #1$}}}\hskip.25ex{}}}
\nc{\ccic}[1]{{}^{\bigcirc\hskip-1.5ex{\raisebox{-0.015cm}{\text{$\scriptscriptstyle #1$}}}\hskip.25ex{}}}
\nc{\ncic}[1]{ {\bigcirc\hskip-1.6ex{\raisebox{-0.0cm}{\text{$\scriptstyle #1$}}}\hskip.25ex{}}}
\nc{\nncic}[1]{ {\bigcirc\hskip-2ex{\raisebox{-0.0cm}{\text{$\scriptstyle #1$}}}\hskip.25ex{}}}
\nc{\cci}[1]{{}_{{}^{ {\textstyle \bigcirc}\hskip-2.05ex{#1}\hskip-.35ex{}}}}
\nc{\ccicc}[1]{{}_{{}^{ {\textstyle \bigcirc}\hskip-1.55ex{#1}\hskip-0.1ex{}}}}
\nc{\x}{\rm{x}}
\nc{\tworow}[2]{\left(#1 \atop #2\right)}

\nc{\fl}{{\mathfrak l}}
\nc{\fm}{{\mathfrak m}}

\begin{document}
%%%%%%%%%%%%%%%%%%%% title %%%%%%%%%%%%%%%%%%%%%%%%%%%%%%%%%%%%%%%%%%%%%%%%
\title{\bf Ramified and Unramified Motivic Multiple $t$-, $T$- and $S$-Values}
\author{
{Ce Xu${}^{a,}$\thanks{Email: cexu2020@ahnu.edu.cn}\ \ and Jianqiang Zhao${}^{b,}$\thanks{Email: zhaoj@ihes.fr}}\\[1mm]
\small a. School of Mathematics and Statistics, Anhui Normal University, Wuhu 241002, PRC\\
\small b. Department of Mathematics, The Bishop's School, La Jolla, CA 92037, USA
%\\[5mm]\emph{\normalsize Dedicated to Professor Masanobu Kaneko on the occasion of his 65th birthday}
}

\date{}
\maketitle
\noindent{\bf Abstract.} In this paper, we consider several variants of motivic multiple zeta values of level two by restricting the summation indices to fixed parity patterns. These variants include Hoffman's multiple $t$-values, Kaneko and Tsumura's multiple $T$-values, and the multiple $S$-values previously studied by the authors. By applying the descent theory of Brown and Glanois to the motivic versions of these values, we derive criteria for determining when they are ramified or unramified. Assuming Grothendieck's period conjecture, our results partially confirm a conjecture by Kaneko and Tsumura regarding the unramified nature of multiple $T$-values of depth less than four. We obtain similar results for motivic multiple $S$-values. Furthermore, we generalize a result of Charlton to broader families of unramified multiple $t$-values with unit components. Finally, we propose several open problems for future research.



%\setcounter{section}{-1}
%\tableofcontents

\section{Introduction}
Let $\N$ be the set of positive integers and $\N_0=\N\cup\{0\}$. For any $d\in\N$ and $(k_1,\dots,k_d)\in\N^d$ we
define the \emph{multiple zeta values} (MZVs) by
\begin{equation*}
\zeta(k_1,\dots,k_d):=\sum_{0<n_1<\cdots<n_d} \frac{1}{n_1^{k_1}\cdots n_d^{k_d}}.
\end{equation*}
We need to impose the condition $k_d\ge 2$ to guarantee convergence and we call the $d$-tuple of positive integers
$\bfk:=(k_1,\dots,k_d)$ \emph{admissible} in this case. Conventionally, $d$ is called the \emph{depth} and
$|\bfk|:=k_1+\cdots+k_d$ is called the \emph{weight}.

MZVs and their higher level generalizations, as an important class of periods in the sense of Kontsevich and Zagier,
have played significant roles in both mathematics and physics in recent years due to their
far-reaching connections with many areas such as algebra geometry, number theory, knot theory and theoretical physics (see e.g. \cite{Broadhurst1996,Brown2012,Deligne2010,DeligneGo2005,SchlottererSt2012,Zhao2016}). At level two, one can define the so-called \emph{Euler sums}
(also called alternating MZVs) by introducing alternating signs on the numerators: for any $\eps_1,\dots,\eps_d=\pm1$,
\begin{equation*}
\zeta(k_1,\dots,k_d;\eps_1,\dots,\eps_d):=\sum_{0<n_1<\cdots<n_d} \frac{\eps_1^{n_1}\cdots\eps_d^{n_d}}{n_1^{k_1}\cdots n_d^{k_d}}.
\end{equation*}
As a convention we put a bar on top of $k_j$ if the corresponding $\eps_j=-1$. For example, $\zeta(\bar1)=\zeta(1;-1)=-\log 2$.

In the mid 1990s, Broadhurst \cite{Broadhurst1996a} already noticed that some of the Euler sums actually lie in the $\Q$-span $\MZV$ of MZVs, and he called these values honorary MZVs. Such a phenomenon is very interesting because it reflects the important Galois descent structure studied recently by Glanois in her thesis \cite{Glanois2015,Glanois2016}. Following her terminology, we call such values unramified.

The \emph{multiple $T$-values} (MTVs) are a class of infinite sums which can be regarded as a level two variation of
the multiple zeta values, first defined by Kaneko and Tsumura \cite{KanekoTs2020}:
\begin{equation*}
T(k_1,\dots,k_d):=\sum_{\substack{0<n_1<\cdots<n_d\\ n_j\equiv j\pmod{2}}} \frac{2^d}{n_1^{k_1}\cdots n_d^{k_d}}=
\sum_{0<n_1<\cdots<n_d} \prod_{j=1}^d \frac{1+(-1)^{n_j-j}} {n_j^{k_j}}
\end{equation*}
for all admissible $(k_1,\dots,k_d)$. We again call $|\bfk|$ its \emph{weight} and $d$ its \emph{depth}. The most important
property of MTVs is that they satisfy the \emph{duality} relation $T(\bfk)=T(\bfk^\ast)$ where the dual
\begin{equation*}
(a_1+1,1_{b_1-1},\ldots,a_r+1,1_{b_r-1})^\ast=(b_r+1,1_{a_r-1},\ldots,b_1+1,1_{a_1-1}).
\end{equation*}
From numerical evidence, they proposed the following conjecture
(cf. \cite[Conjecture 5.2]{KanekoTs2020}).

\begin{conj}
1) For even weights, $T(p,q,r)$ with odd $p, r\ge 3$ and
even $q$ (and their duals) are in $\MZV$. Together with the single $T$-values
these (and their duals) are the only even weight unramified MTVs (i.e., contained in $\MZV$).

2) For odd weights, $T(p,1,r)$ with even $p,r$ (and their duals) are in $\MZV$.
Together with the single and double $T$-values these (and their duals) are the
only odd weight unramified MTVs.
\end{conj}

In \cite{Murakami2021}, Murakami proved that the conditions in the above conjecture are sufficient by using the motivic theory of Euler sums developed in \cite{Brown2012,Glanois2016}. In this paper, we will use this machinery to further show that,
on the motivic level, these conditions are also necessary, at least when the depth is bounded by three.
This would completely confirm the above conjecture for depth $\le 3$ if we assume Grothendieck's period conjecture.

\begin{thm} \label{thm:MTV}
Suppose $\bfk$ has depth at most three. We have the following necessary and sufficient conditions for $T^\fm(\bfk)$ to be unramified.
\begin{enumerate}
\item [\upshape{(1)}] If $\dep(\bfk)=1$, then $\bfk=k\ge 2$.

\item [\upshape{(2)}] If $\bfk=(k,l)$ then $l\ge 2$ with $k+l$ odd.

\item [\upshape{(3)}] If $\dep(\bfk)=3$ then either
\begin{enumerate}
 \item [\upshape{(i)}] $\bfk=(1,1,2), (1,1,3), (1,2,2)$ or
 \item [\upshape{(ii)}] $|\bfk|$ is even, $\bfk=(O_{>1},E,O_{>1})$, or
 \item [\upshape{(iii)}] $|\bfk|$ is odd, $\bfk=(E,1,E)$,
\end{enumerate}
where $O$ (resp.~$E$) may represent different odd (resp.~even) numbers.
\end{enumerate}
\end{thm}

Case (1) of Thm.~\ref{thm:MTV} is obvious. We will prove case (2) in Thm.~\ref{thm:DTDS}. Proof of case (3) will be broken into many subcases throughout \S\ref{sec:TripleTspecial} and \S\ref{sec:TripleTandS}.

Beside MTVs, there are other variants of MZVs of level two. For any admissible $\bfk=(k_1,\dots,k_d)$, Hoffman defines the \emph{multiple $t$-values} in \cite{Hoffman2019} by
\begin{align*}
t(\bfk):=\sum_{\substack{0<n_1<\cdots<n_d\\ n_j: \text{ odd}}}
 \frac{2^d}{n_1^{k_1}\cdots n_d^{k_d}}=
\sum_{0<n_1<\cdots<n_d} \prod_{j=1}^d \frac{1-(-1)^{n_j}} {n_j^{k_j}}.
\end{align*}
The current authors defined the \emph{multiple $S$-values} $S(\bfk)$ in \cite{XuZhao2020a} by
\begin{align*}
S(\bfk):=\sum_{\substack{0<n_1<\cdots<n_d\\ n_j\equiv j+1\pmod{2}}} \frac{2^d}{n_1^{k_1}\cdots n_d^{k_d}}=
\sum_{0<n_1<\cdots<n_d} \prod_{j=1}^d \frac{1-(-1)^{n_j-j}} {n_j^{k_j}}.
\end{align*}

\begin{rem}
a. It is clear from the definitions that all MTVs, MtVs and MSVs are $\Z$-linear combinations of Euler sums.

b. All the above values are special cases of the \emph{multiple mixed values} first studied by the current authors in \cite{XuZhao2020a}, the definition of which allows all possible parity patterns in the infinite sums. In depth two, however, there are only three kinds of double mixed values that are not automatically MZVs: double $t$-, $T$- and $S$-values.
\end{rem}

Applying the same technique for MTVs and alongside with the proof of Thm.~\ref{thm:MTV}, we will prove the following results about multiple $S$-values in \S\ref{sec:TripleTspecial} and \S\ref{sec:TripleTandS}.

\begin{thm} \label{thm:MSV}
Suppose $\bfk$ has depth at most three. We have the following necessary and sufficient conditions for $S^\fm(\bfk)$ to be unramified.
\begin{enumerate}
\item [\upshape{(1)}] If $\dep(\bfk)=1$, then $\bfk=k\ge 2$.

\item [\upshape{(2)}] If $\bfk=(k,l)$ then $k\ge 2$ and $l\ge 2$ with $k+l$ odd.

\item [\upshape{(3)}] If $\dep(\bfk)=3$ then either
\begin{enumerate}
 \item [\upshape{(i)}] $|\bfk|$ is even, $\bfk=(O_{>1},O,E)$ or $\bfk=(O_{>1},E,O_{>1})$, or
 \item [\upshape{(ii)}] $|\bfk|$ is odd, $\bfk=(E,1,E)$,
\end{enumerate}
where $O$ (resp.~$E$) may represent different odd (resp.~even) numbers.
\end{enumerate}
\end{thm}

On the other hand, there are many unramified MtVs as manifested by the results of Murakami and Charlton.
We denote by $\{\bfs\}_a$ the string obtained by repeating $\bfs$ exactly $a$ times. If $\bfs$ has just one
component then we often drop the curly brackets for simplicity.

\begin{thm} \label{thm:unramfiedMtV-MC} The MtV $t(\bfk)$ is unramified if either one of the following conditions holds:
\begin{enumerate}
\item [\upshape{(1)}] \emph{(\hskip-1pt\cite[Thm.~1]{Murakami2021})} all components of $\bfk$ are at least $2$;
\item [\upshape{(2)}] \emph{(\hskip-1pt\cite[Prop.\ 5.11]{Charlton2021})} $\bfk=(2_a, 1, 2_b, 2n + 1, 2_c)$, $a,n\in\N$ and $b,c\in\N_0.$
\end{enumerate}
\end{thm}

We can extend the above to more cases of unramified MtVs in the next theorem.

\begin{thm}\label{thm:unramfiedMtV}
 The MtV $t(\bfk)$ is unramified if either one of the following conditions holds:
\begin{enumerate}
\item [\upshape{(1)}] $\bfk=(\{2n\}_e,1,\bfq)$, where $n,e\in\N$ and
all components of $\bfq$ are at least $2$;
\item [\upshape{(2)}] $\bfk=(2n,2m,2n,1,\bfq)$, where $m,n\in\N$ and all components of $\bfq$ are at least $2$;
\item [\upshape{(3)}] $\bfk=(2,1,3,2,1,\bfq)$, where all components of $\bfq$ are at least $2.$
\end{enumerate}
\end{thm}

In fact, we prove the above three cases in Thm.~\ref{thm:unramfiedMtV1}, Thm.~\ref{thm:unramfiedMtV2}, and Thm.~\ref{thm:unramfiedMtV3}, respectively. Then Thm.~\ref{thm:unramfiedMtV} follows immediately by applying the period map.

Note that if a unit component (i.e. component equals to 1) appears at either end of a motivic MtV then it is ramified (see Lemma~\ref{lem:Charlton}). Our next result provides
a family of ramified motivic MtVs neither starting nor ending with a unit component.

\begin{thm}\label{thm:ramfiedMtV} \emph{(Thm.~\ref{thm:ramfiedMtV1})}
Suppose $j\le n<\ell$ such that $k_j\ge 3$ is odd and $k_i$ is even for all $1\le i<n$ and $i\ne j$.
Suppose further that $k_n=1$ and $k_i\ge 2$ for all $n< i\le \ell$.
Then the motivic MtV $t^\fm(k_1,\dots,k_\ell)$ is ramified.
\end{thm}

Many results on motivic MZVs have been found in the past decade. Some can be found in the works of  \cite{Brown2014,Brown2021,Charlton2021,CharltonHoffman-MathZ2025,CharltonKeilthy2024,JinLi2020,Keilthy2022,Li2019,Li2024,LiLiu2021,Murakami2021,SchlottererSt2012} and references therein.

\section{Notation and the motivic setup}\label{sec:setup}
We will adopt the motivic setup developed in \cite{Brown2012,Glanois2016}. Let $\Gamma_{N}$ be the group of $N$-th roots of unity for any $N\in\N$. Let $w\in\N$ and suppose $a_j=0$ or $a_j\in\Gamma_{N}$ for all $0\le j\le w+1$. Then one can define the so-called motivic integrals
$I^\fm(a_0;a_1,\dots,a_w;a_{w+1})$ subjecting to a list of axioms.
The most important property of the motivic integrals is that there is a period map $\dch$ such that
\begin{equation}\label{equ:periodMap}
\dch \big(I^\fm(a_0;a_1,\dots,a_w;a_{w+1})\big)=(-1)^{\dep(a_1,\dots,a_w)}
\int_{a_0}^{a_{w+1}} \frac{dt}{t-a_1} \cdots \frac{dt}{t-a_w}
\end{equation}
as an iterated integral (we always integrate from left to right in this paper)
whenever it converges in which case we can obtain some special values of multiple polylogarithms at some $N$th roots of unity (see \cite[Ch.\ 14]{Zhao2016}). Here $\dep(a_1,\dots,a_w)$ is the number of nonzero components in $(a_1,\dots,a_w)$.

Let $\calH^N$ be the $\Q$-vector space spanned by the \emph{motivic colored MZVs} of the form
\begin{equation*}
\zeta^\fm_c\tworow{k_1,\dots,k_d}{\eps_1,\dots,\eps_d}=\zeta^\fm_c(k_1,\dots,k_d;\eps_1,\dots,\eps_d):= I^\fm(0;0_c,\eta_1,0_{k_1-1},\dots,\eta_d,0_{k_d-1};1),
\end{equation*}
where $c\in\N_0$, $k_j\in\N, \eps_j\in\Gamma_{N}$, $\eta_j=1/\eps_j\cdots \eps_d$ for all $j=1,\dots,d$.
We simply write $\zeta^\fm$ for $\zeta^\fm_0$.
These are the motivic MZVs when $N=1$ and are the motivic Euler sums when $N=2$. Here, we use both two-row and one-row notation for Euler sums to save space at various places.

We now list all the important properties/axioms of the motivic integrals
as follows (cf. \cite[\S2.4]{Brown2012}, \cite[p.~9]{Glanois2016}, and \cite[\S2, (I1)-(I6)]{Murakami2021}):
\begin{itemize}
	\item[(I1)] Empty word: $I^\fm(a_{0}; a_{1})=1$.

	\item[(I2)] Trivial path: $\forall$ weight $n\ge1$, $I^\fm(a_{0}; a_{1}, \dots, a_{n}; a_{n+1})=0$ if $a_{0}=a_{n+1}$.

	\item[(I3)] Shuffle product: $\forall c\in\N_0$ we have
	\begin{equation*}
\zeta_c^\fm \tworow{k_1,\dots,k_d}{\eps_1,\dots,\eps_d}=
(-1)^c\sum_{\substack{i_{1}+ \cdots + i_d=c\\ i_{1},\dots,i_d\ge0}} \left(\prod_{j=1}^d\binom {k_j+i_j-1} {i_j}\right) \zeta^\fm \tworow{k_1+i_1, \cdots , k_d+i_d}{\eps_1\ \ , \cdots ,\ \ \eps_d}.
 \end{equation*}

	\item[(I4)] Regularization: If $a_{1}=\cdots=a_n\in \{0,1\}$, then
	$I^\fm(0; a_{1}, \cdots, a_{n}; 1)=0.$

	\item[(I5)] Path reversal: $I^\fm(a_{0}; a_{1}, \cdots, a_{n}; a_{n+1})= (-1)^n I^\fm(a_{n+1}; a_{n}, \cdots, a_{1}; a_{0}).$

	\item[(I6)] Homothety: $\forall \alpha \in \Gamma_{N}, I^\fm(0; \alpha a_{1}, \cdots, \alpha a_{n}; \alpha a_{n+1}) = I^\fm(0; a_{1}, \cdots, a_{n}; a_{n+1})$.

	\item[(I7)] Change of variable $t\to 1-t: \forall a_{1}, \cdots, a_{n}\in\{0, 1\},$ $$I^\fm(0; a_{1}, \cdots, a_{n}; 1)= I^\fm(0;1-a_n, \cdots, 1-a_1; 1).$$

	\item[(I8)] Path composition: $\forall a,b, x\in \Gamma_{N} \cup \left\{0\right\}$,
	$$ I^\fm(a; a_{1}, \cdots, a_{n}; b)=\sum_{i=0}^{n} I^\fm(a; a_{1}, \cdots, a_{i}; x) I^\fm(x; a_{i+1}, \cdots, a_{n}; b) .$$
\end{itemize}
We remark that the regularization procedure of divergent MZVs (whose motivic integral satisfies the condition $a_n=a_{n+1}$) is directly related to the shuffle product formula (I3) after one applies (I7).

Set $\calA^N=\calH^N/\zeta^\fm(2)\calH^N$. Denote by
$\calH^N_w$ and $\calA_w^N$ the weight $w$ part for all $w\ge 0$. Let $\calL^N=\calA^N_{>0}/\calA^N_{>0}\cdot \calA^N_{>0}$.
For any weight $w$ and odd $r$ such that $0<r<w$ by modifying the coproduct of certain Hopf algebra
one can define a derivation as part of a coaction
$$
D_r: \calH_w^N \to \calL_r^N \ot \calH_{w-r}^N
$$
by sending $I^\fm(a_0;a_1,\dots,a_w;a_{w+1})$ to
$$
\sum_{p=0}^{w-r} I^\fl(a_p; a_{p+1},\dots,a_{p+r};a_{p+r+1})\ot I^\fm(a_0;a_1,\dots,a_p,a_{p+r+1},\dots,a_w;a_{w+1}).
$$
The sequence in the left motivic integral is called a \emph{subsequence} of $(a_0;a_1,\dots,a_w;a_{w+1})$ while
that in the right is called a \emph{quotient sequence}. Each such a choice is called a \emph{cut}.

 From now on, we only consider the case of $N = 2$, which is about the motivic Euler sums.
For convenience, we will write
\begin{equation*}
     \zeta_c^\fm (k_1,\dots,k_d):=\zeta_c^\fm \tworow{k_1,\dots,k_d}{1,\dots,1}.
\end{equation*}
Further, if $\eps_1,\dots,\eps_d=\pm$ then we will decorate $\zeta_c^\fm (k_1,\dots,k_d)$
by putting a bar on top of $k_j$ if $\eps_{j}\eps_{j+1}=-1$ (here $\eps_{d+1}=1$).
For example, the motivic version of $\log 2$ is defined by
\begin{equation*}
\log^\fm 2:=-I^\fm(0; -1;1)=-\zeta^\fm (\bar 1).
\end{equation*}

Recall the following theorem which combines
Glanois's result \cite[Cor.\ 2.4]{Glanois2016} with Brown's \cite[Thm.\ 3.3]{Brown2012}.
Set $\calH=\calH^1$ and $\calL=\calL^1$.

\begin{thm}\label{thm-Glanois}
Let $a\in\N_0$, $\bfeps\in\{\pm 1\}^d$, and $\bfk\in\N^d$ such that $a+|\bfk|=w$.
Then the weight $w$ motivic Euler sum $\zeta_a^\fm\tworow{\bfk}{\bfeps}\in\calH_w$ if and only $D_1 \zeta_a^\fm\tworow{\bfk}{\bfeps}=0$ and
$D_r \zeta_a^\fm\tworow{\bfk}{\bfeps}\in \calL_r\ot \calH_{w-r}$ for all odd $r<w$.
Moreover, if $D_r \zeta_a^\fm\tworow{\bfk}{\bfeps}=0$ for all odd $r<w$ then
$\zeta_a^\fm\tworow{\bfk}{\bfeps}=c \, \zeta^\fm(w)$ for some rational number $c$.
\end{thm}

Recently, Charlton, Hoffman, and Sato \cite{CharltonHoffmanSato-arxiv2026} gave an explicit formula for the Galois descent that expresses multiple \(t\)-values of maximal height in terms of classical multiple zeta values, thereby making precise Murakami's earlier motivic result.

The following result is due to Glanois. A complete proof is given in \cite[Lemma 2.2]{XuZhao2023Aug}.
\begin{lem} \label{lem:D1}
Suppose every component of $\bfk$ is a positive integer possibly decorated by a bar.
If $\bar1$ does not appear in $\bfk$ then $D_1\zeta_a^\fm(\bfk)=0$ for all $a\in\N_0$.
\end{lem}

For all $a\in\N_0$ and $\bfk=(k_1,\dots,k_d)\in\N^d$ we define the motivic multiple $t$-, $T$-, $S$- and $\tz$-values by
\begin{align}\label{defn:MMtV}
t^\fm_a(\bfk)=&\, \sum_{\eta_1,\dots,\eta_d=\pm 1} \eta_1 I^\fm(0;0_a,\eta_1,0_{k_1-1},\dots,\eta_d,0_{k_d-1};1),\\
T^\fm_a(\bfk)=&\, \sum_{\eta_1,\dots,\eta_d=\pm 1} \eta_1\cdots\eta_d I^\fm(0;0_a,\eta_1,0_{k_1-1},\dots,\eta_d,0_{k_d-1};1),\label{defn:MMTV} \\
S^\fm_a(\bfk)=&\, \sum_{\eta_1,\dots,\eta_d=\pm 1} \eta_2\cdots\eta_d I^\fm(0;0_a,\eta_1,0_{k_1-1},\dots,\eta_d,0_{k_d-1};1), \label{defn:MMSV} \\
\tz^\fm_a(\bfk)=&\, \sum_{\eta_1,\dots,\eta_d=\pm 1} I^\fm(0;0_a,\eta_1,0_{k_1-1},\dots,\eta_d,0_{k_d-1};1), \label{defn:MMZV}
\end{align}
respectively. Setting $f^\fm=f^\fm_0$ for $f=t, T, S$, it is easy to see that for all admissible $\bfk$ we have
\begin{equation*}
\dch T^\fm(\bfk)=T(\bfk), \quad \dch t^\fm(\bfk)=t(\bfk),\quad
\dch S^\fm(\bfk)=S(\bfk).
\end{equation*}
Further, by the motivic version of the distribution relation proved by Charlton we see that
$\tz^\fm_a(\bfk)=2^{\dep(\bfk)-a-|\bfk|}\zeta^\fm_a(\bfk)$ if $\bfk$ is admissible.
For future reference, we state it below.

\begin{thm}\label{thm-distribution} \emph{(cf. \cite[Cor.~5.4]{Charlton2021})}
Let $a\in\N_0$ and assume $\bfk$ is admissible. Then we have
\begin{equation}\label{equ:distribution}
\tz^\fm_a(\bfk)=2^{\dep(\bfk)-a-|\bfk|}\zeta^\fm_a(\bfk).
\end{equation}
\end{thm}

We will use the next technical lemma of \cite{Murakami2021} a number of times in this paper.
\begin{lem}\label{Murakami-lemma9} \emph{(cf. \cite[Lemma 9]{Murakami2021})}
Suppose $a_0,\dots,a_d\ge 1$. For any fixed $\ga,\gb\in\{\pm1\}$ we have
\begin{equation*}
\sum_{\eta_1,\dots,\eta_d=\pm1} I^\fl(\ga;0_{a_0-1},\eta_1,0_{a_1-1},\dots,\eta_d,0_{a_d-1};\gb)=0.
\end{equation*}
\end{lem}

\section{Some preliminary results}\label{sec:preliminary}
To save space, we extend $I^\fl(a;--;b)$ and $I^\fm(a;--;b)$ $\Q$-linearly to all finite $\Q$-linear combinations of strings of 0's and $\pm1$'s. For example,
\begin{equation*}
I^\fm(1;011+0\bar10;0)=I^\fm(1;011;0)+I^\fm(1;0\bar 10;0).
\end{equation*}
Here and in the rest of the paper, as our convention for Euler sums we will often write $\bar 1$ for the component $-1$ in the motivic iterated integrals $I^\fm$.

The next lemma follows from definition (I3) in \S \ref{sec:setup} immediately.
\begin{lem}\label{lem:zetaBar}
For all integers $a\ge0$ and $k\ge 2$ we have
\begin{align*}
\zeta^\fm_a(k)=&\, (-1)^a \binom{a+k-1}{a} \zeta^\fm(a+k), \\
\zeta^\fm_a(\ol{k})=&\, (-1)^a (2^{1-a-k}-1)\binom{a+k-1}{a} \zeta^\fm(a+k).
\end{align*}
\end{lem}
\begin{proof}
By (I3) we only need to show that $\zeta^\fm(\ol{n})=(2^{1-n}-1)\zeta^\fm(n)$ for all $n\ge 2$.
It is easy to see that for all odd $r<n$ and $r>1$ we have $D_r \zeta^\fm(\ol{n})=0$.
Combining with Lemma \ref{lem:D1} we see that $\zeta^\fm(\ol{n})=c \zeta^\fm(n)$
for some $c\in\Q$. The lemma follows if we apply the period map $\dch$ of \eqref{equ:periodMap}
and use the fact the $\zeta(\ol{n})=(2^{1-n}-1)\zeta(n)$.
\end{proof}

\begin{lem}\label{lem:singleTs}
Suppose $a\in\N_0$, $k\ge 1$, and $a+k\ge 2$. For motivic single $t$-, $T$- and $S$-values we have
\begin{align*}
t^\fm(1)=T^\fm(1)=&\, \log^\fm 2, \qquad S^\fm(1)=-\log^\fm 2,\\
\quad t^\fm_a(k)=T^\fm_a(k)=&\, (-1)^a (2-2^{1-a-k})\binom{a+k-1}{a} \zeta^\fm(a+k), \\
S^\fm_a(k)=&\, (-1)^a 2^{1-a-k}\binom{a+k-1}{a} \zeta^\fm(a+k).
\end{align*}
\end{lem}

\begin{proof} By definition
\begin{align*}
 t^\fm(1)=T^\fm(1)= &\, I^\fm(0;1;1)-I^\fm(0;\bar1;1)=\log^\fm 2, \\
S^\fm(1)= &\, I^\fm(0;1;1)+I^\fm(0;\bar1;1)=-\log^\fm 2,
\end{align*}
and for all $a\ge 0$ and $k\ge 1$ with $a+k\ge 2$
\begin{align*}
 t^\fm_a(k)=T^\fm_a(k)= &\, I^\fm(0;0_{a}10_{k-1}-0_{a}\bar1,0_{k-1};1)=\zeta_a^\fm(k)-\zeta_a^\fm(\ol{k}) \\
= &\, (-1)^a (2-2^{1-a-k})\binom{a+k-1}{a} \zeta^\fm(a+k), \\
S^\fm_a(k)= &\, \zeta_a^\fm(k)+\zeta_a^\fm(\ol{k}) =(-1)^a 2^{1-a-k}\binom{a+k-1}{a} \zeta^\fm(a+k)
\end{align*}
by Lemma~\ref{lem:zetaBar}.
\end{proof}

Throughout the paper, we use the generalized Kronecker symbol $\delta_{\text{conditions}}=1$ if all the conditions are true
and $\delta_{\text{conditions}}=0$ otherwise.

\begin{lem}\label{lem:D1S1bfl}
Assume $\bfk=(k_1,\dots,k_p)$ is admissible.
 \begin{enumerate}
 \item [\upshape{(1)}] $D_1 T^\fm(\bfk)=0$.

 \item [\upshape{(2)}] $D_1 S^\fm(\bfk)=-2\delta_{k_1=1} \log^\fl 2\ot T^\fm(k_2,\dots,k_p) $.
Consequently, $S^\fm(1,\bfl)$ are ramified for all nontrivial $\bfl$.
\end{enumerate}
\end{lem}
\begin{proof}
Part (1) of the lemma is \cite[Lemma~43]{Murakami2021} by Murakami.
By the same argument, we see that in general
\begin{equation*}
D_1 S^\fm(\bfk)= \delta_{k_1=1} \sum_{\eta_1,\dots,\eta_p=\pm1}
 (\eta_2-\eta_1)\eta_3\dots\eta_p I^\fl(0;\eta_1;\eta_2)\ot I^\fm(0;\eta_2,0_{k_2-1},\dots,\eta_p,0_{k_p-1};1)
\end{equation*}
for $\bfk=(k_1,\dots,k_p)$. Thus $D_1 S^\fm(\bfk)=0$ if $k_1\ge 2$. On the other hand,
if $k_1=1$ then we see that
\begin{equation*}
D_1 S^\fm(\bfk)= -2 \log^\fl 2\ot T^\fm(k_2,\dots,k_p) \ne 0
\end{equation*}
which implies that $S^\fm(1,k_2,\dots,k_p)$ is ramified by Thm.~\ref{thm-Glanois}. This proves (2).
\end{proof}

\section{Reduction formula for double Euler sums of odd weight}\label{sec:reductionDB}
In this section, we will express every motivic double Euler sum of odd weight as a $\Q$-linear combination
of products of motivic Riemann zeta values (after setting $\zeta^\fm(0)=-1/2$). As corollaries, we can
write all the elements $t^\fl(k,l)$, $T^\fl(k,l)$ and $S^\fl(k,l)$ in $\calL$
as rational multiples of $\zeta^\fl(k+l)$ when the weight $k+l$ is odd.

First, we generalize a result of Zagier \cite[Prop.~7]{Zagier2012} concerning motivic double zetas to motivic double Euler sums. The non-motivic version can be found in \cite[(74)]{BorweinBrBr1997}.
To save space, we use one row notation for Euler sums in this section.

\begin{prop}\label{prop:DBESreduce}
Let $m,n\in\N$ and $\eta_1,\eta_2=\pm1$. Put $\zeta(0;\pm1)=-1/2$.
Then the double zeta value $\zeta^\fm(m,n;\eta_1,\eta_2)$ of weight
$m + n = w = 2K + 1$ is given by
\begin{align*}
\zeta^\fm(m,n;\eta_1,\eta_2)=&\, -\frac12\zeta^\fm(w;\eta_1\eta_2) +
(-1)^{m}\sum_{s=0}^{K}\bigg[\binom{w-2s-1}{m-1} \zeta^\fm(w-2s;\eta_1)\zeta^\fm(2s;\eta_1\eta_2)\\
 &\, +\binom{w-2s-1}{n-1}\zeta^\fm(w-2s;\eta_2)\zeta^\fm(2s;\eta_1\eta_2)\bigg]
+ \gd_{2|n}\, \zeta^\fm(m;\eta_1)\zeta^\fm(n;\eta_2),
\end{align*}
where $2|n$ means $n$ is an even number.
\end{prop}

\begin{proof}
We will follow Zagier's idea in the proof of \cite[Prop.~7]{Zagier2012}. For $\eta_1,\eta_2=\pm1$ and $w\ge 3$ we define the generating functions
\begin{align*}%\label{equ:PandD}
P^\fm_{\eta_1,\eta_2}(x,y):=&\,\sum_{r,s\in\N, r+s=w} \zeta^\fm\binom{r}{\eta_1}\zeta^\fm\binom{s}{\eta_2} x^{r-1}y^{s-1}, \\
D^\fm_{\eta_1,\eta_2}(x,y):=&\,\sum_{r,s\in\N, r+s=w} \zeta^\fm\binom{r\, ,\,s}{\eta_1,\eta_2} x^{r-1}y^{s-1}.
\end{align*}
Then we can derive
\begin{align*}
P^\fm_{\eta_1,\eta_2}(x,y)= &\,D^\fm_{\eta_1,\eta_2}(x,y)+D^\fm_{\eta_2,\eta_1}(y,x)
    +\zeta^\fm(w;\eta_1\eta_2)F_w(x,y)\\
 = &\,D^\fm_{\eta_1\eta_2,\eta_2}(x,x+y)+D^\fm_{\eta_1\eta_2,\eta_1}(y,x+y),
\end{align*}
where $F_w(x,y)=\frac{x^{w-1}-y^{w-1}}{x-y}$.
Setting $(\eta_1,\eta_2)=(-1,-1)$ and $(-1,1)$
\begin{align*}
P^\fm_{-,-}(x,y)= &\,D^\fm_{-,-}(x,y)+D^\fm_{-,-}(y,x) +\zeta^\fm(w)F_w(x,y)\\
 = &\,D^\fm_{+,-}(x,x+y)+D^\fm_{+,-}(y,x+y),\\
P^\fm_{-,+}(x,y)= &\,D^\fm_{-,+}(x,y)+D^\fm_{+,-}(y,x) +\zeta^\fm(\ol{w})F_w(x,y)\\
 = &\,D^\fm_{-,-}(x,x+y)+D^\fm_{-,+}(y,x+y)
\end{align*}
and repeatedly applying the above two equations, we get (setting $z=x-y$)
\begin{align*}
&\, D^\fm_{-,+}(x,y)\\
=&\, P^\fm_{-,+}(x,y)-\zeta^\fm(\ol{w})F_w(x,y)-D^\fm_{+,-}(y,x) \\
=&\, P^\fm_{-,+}(x,y)-\zeta(\ol{w})F_w(x,y)-P^\fm_{-,-}(y,z)+D^\fm_{+,-}(z,x) \\
=&\, P^\fm_{-,+}(x,y)-\zeta(\ol{w})F_w(x,y)-P^\fm_{-,-}(y,z)+P^\fm_{-,+}(x,z)-\zeta^\fm(\ol{w})F_w(z,x)-D^\fm_{-,+}(x,z) \\
=&\, P^\fm_{-,+}(x,y)-\zeta(\ol{w})F_w(x,y)-P^\fm_{-,-}(y,z)+P^\fm_{-,+}(x,z)-\zeta^\fm(\ol{w})F_w(z,x) \\ &\,-P^\fm_{-,+}(x,-y)+D^\fm_{-,-}(-y,z) \\
=&\, P^\fm_{-,+}(x,y)-\zeta(\ol{w})F_w(x,y)-P^\fm_{-,-}(y,z)+P^\fm_{-,+}(x,z)-\zeta^\fm(\ol{w})F_w(z,x) \\
 &\, -P^\fm_{-,+}(x,-y)+P^\fm_{-,-}(-y,z)-\zeta^\fm(w)F_w(-y,z)-D^\fm_{-,-}(z,-y) \\
=&\, P^\fm_{-,+}(x,y)-\zeta(\ol{w})F_w(x,y)-P^\fm_{-,-}(y,z)+P^\fm_{-,+}(x,z)-\zeta^\fm(\ol{w})F_w(z,x)\\
 &\,-P^\fm_{-,+}(x,-y)+P^\fm_{-,-}(-y,z)-\zeta^\fm(w)F_w(-y,z)-P^\fm_{-,+}(-x,z)+D^\fm_{-,+}(-x,-y)\\
=&\, P^\fm_{-,+}(x,y)-\zeta(\ol{w})F_w(x,y)+P^\fm_{-,-}(-y,y-x)+P^\fm_{-,+}(x,z)-\zeta^\fm(\ol{w})F_w(z,x)\\
 &\,+P^\fm_{-,+}(-x,y)+P^\fm_{-,-}(-y,z)-\zeta^\fm(w)F_w(-y,z)+P^\fm_{-,+}(x,y-x)-D^\fm_{-,+}(x,y)\\
=&\, Q^\fm_{-,+}(x,y)+Q^\fm_{+,-}(z,x)+Q^\fm_{-,-}(z,-y)-D^\fm_{-,+}(x,y),
\end{align*}
where
\begin{equation*}
Q^\fm_{\eta_1,\eta_2}(x,y)=P^\fm_{\eta_1,\eta_2}(x,y)+P^\fm_{\eta_1,\eta_2}(-x,y)-\zeta^\fm(w;\eta_1\eta_2)F_w(x,y).
\end{equation*}
Similarly,
\begin{align*}
2D^\fm_{+,-}(x,y)=&\,Q^\fm_{+,-}(x,y)+Q^\fm_{+,-}(z,-y)+Q^\fm_{-,-}(z,x).
\end{align*}
In general
\begin{align*}
2D^\fm_{\eta_1,\eta_2}(x,y)=&\,Q^\fm_{\eta_1,\eta_2}(x,y)+Q^\fm_{\eta_1,\eta_1\eta_2}(z,-y)+Q^\fm_{\eta_2,\eta_1\eta_2}(z,x).
\end{align*}
Hence, putting $c_m=1-(-1)^m$ for all $m$ we get
\begin{align*}
[x^{m-1}y^{n-1}]Q^\fm_{\eta_1,\eta_2}(x,y)=&\, c_m\zeta^\fm(m;\eta_1)\zeta^\fm(n;\eta_2)-\zeta^\fm(w;\eta_1\eta_2),\\
[x^{m-1}y^{n-1}]Q^\fm_{\eta_1,\eta_1\eta_2}(z,-y)
=&\, \sum_{r=m}^{w-1} \binom{r-1}{m-1} (-1)^{n} \Big[\zeta^\fm(w;\eta_2)-c_r\zeta^\fm(r;\eta_1)\zeta^\fm(w-r;\eta_1\eta_2)\Big],\\
[x^{m-1}y^{n-1}]Q^\fm_{\eta_2,\eta_1\eta_2}(z,x)
=&\,\sum_{r=n}^{w-1} \binom{r-1}{n-1} (-1)^{n} \Big[\zeta^\fm(w;\eta_1)-c_r\zeta^\fm(r;\eta_2)\zeta^\fm(w-r;\eta_1\eta_2)\Big].
\end{align*}
Therefore, the proposition follows by the substitution $r=w-2s$ and the identities
\begin{equation*}
 \sum_{r=m}^{w-1} \binom{r-1}{m-1} =\binom{w-1}{m}=\binom{w-1}{n-1}, \qquad\sum_{r=n}^{w-1} \binom{r-1}{n-1} =\binom{w-1}{n}=\binom{w-1}{m-1}.
\end{equation*}
\end{proof}

\begin{cor} \label{cor:DBreduce}
Let $m,n\in\N$ with $m + n = w$ odd. Then we have
\begin{align*}
 \zeta^\fl(m,n;\eta_1,\eta_2)&\,=-\frac12\zeta^\fl(w;\eta_1\eta_2)-\frac12
(-1)^{m}\bigg[\binom{w-1}{m-1} \zeta^\fl(w;\eta_1)+\binom{w-1}{n-1}\zeta^\fl(w;\eta_2)\bigg], \\
\tz^\fl(m,n)&\, = -\bigg(1+(-1)^m\binom{w}{m}\bigg)\tz^\fl(w), \quad
 t^\fl(m,n)=-t^\fl(w),\\
T^\fl(m,n)&\,= (-1)^{n} \binom{w-1}{n} T^\fl(w), \quad  S^\fl(m,n)= (-1)^{n} \binom{w-1}{m} T^\fl(w).
\end{align*}
\end{cor}

\begin{rem}
The formula for $\zeta^\fl(m,n;\eta_1,\eta_2)$ is depth two case of a more general reduction
formula of Glanois's in \cite[Cor.~4.2.4]{Glanois2015}.
\end{rem}

\section{Motivic double $t$-, $T$- and $S$-values}\label{sec:DBMixedVs}
In this section, we present the precise criterion for when a motivic double mixed value is unramified.
First, we recall the following result of Murakami.
\begin{lem}\label{lem:DTwtOdd} \emph{(cf. \cite[Thm.~42]{Murakami2021})}
For integers $m,k\ge 1$, $l\ge 2$ with $k+l+m$ odd, the motivic $T$-value $T^\fm_m(k,l)$ is unramified.
\end{lem}

\begin{prop}\label{prop:DtAllwt}
For $k,l\in\N$ the motivic double $t$-value $t^\fm(k,l)$ is unramified if and only if $k,l\ge 2$.
\end{prop}
\begin{proof}
This follows immediately from Prop.~\ref{prop:MtV1s} and Murakami's result in Lemma~\ref{lem:MtVall>1}.
\end{proof}

\begin{thm} \label{thm:DTDS}
Let $k,l\in\N$. The motivic double $T$-value $T^\fm(k,l)$ is unramified if and only if $k\ge 1$ and $l\ge 2$ with odd weight $k+l$.
The motivic double $S$-value $S^\fm(k,l)$ is unramified if and only if $k,l\ge 2$ with odd weight $k+l$.
\end{thm}
We break the proof of the theorem into the next two propositions according to the weight parity.

\begin{prop}\label{prop:DTwtOdd}
Suppose $k,l\in\N$ have different parities. Then we have
\begin{enumerate}
\item [\upshape{(1)}] $T^\fm(k,1)$, $S^\fm(k,1)$ and $S^\fm(1,k)$ are ramified for all even $k\in\N$, and

\item [\upshape{(2)}] $T^\fm(k,l)$ is unramified if $l\ge 2$ and $S^\fm(k,l)$ is unramified if $k,l\ge 2$.
\end{enumerate}
\end{prop}

\begin{proof}
Let $a=k-1$ and $b=l-1$. We will prove the proposition by computing the derivations $D_r$ for all odd $r<k+l$.

\medskip
\noindent
I. First we consider $D_1$.

(i). If $a=0$ then $b>0$. Then by Lemma~\ref{lem:D1S1bfl}
\begin{equation}\label{equ:D1T1lS1l}
D_1 T^\fm(1,l)=0, \quad\text{and}\quad
D_1 S^\fm(1,l)=- 2 \log^\fl 2 \ot T^\fm(l)\ne 0.
\end{equation}
Thus $S^\fm(1,l)$ is ramified.

(ii). If $a\ne 0$. Then there are no nontrivial cuts for $D_1 T^\fm(k,l)$ and $D_1 S^\fm(k,l)$ except when $b=0$.
If $b=0$ then
\begin{align}\label{equ:D1Tk1}
D_1 T^\fm(k,1)=&\, \sum_{\eta_1,\eta_2} \eta_1\eta_2 I^\fl(0;\eta_2;1)\ot I^\fm(0;\eta_1,0_a;1)=\log^\fl 2 \ot T^\fm(k)\ne0,\\
D_1 S^\fm(k,1)=&\, \sum_{\eta_1,\eta_2} \eta_2 I^\fl(0;\eta_2;1)\ot I^\fm(0;\eta_1,0_a;1)=\log^\fl 2 \ot S^\fm(k)\ne 0. \label{equ:D1Sk1}
\end{align}
Thus $T^\fm(k,1)$ and $S^\fm(k,1)$ are both ramified.

\medskip
\noindent
II. Suppose $a=0,b\ge 1$, $3\le r\le b$, and both $b$ and $r$ are odd.
Then there is only one possible cut for $D_rT^\fm(1,b+1)$ shown in the following picture.

\begin{center}
\begin{tikzpicture}[scale=0.9]
\node (A0) at (0.05,0) {$0;\eta_1,\eta_2,0_{r-1},0,0_{b-r};1$};
\node (A1) at (-.5,-0.4) {${}$};
\draw (-1.6,-0.25) to (-1.6,-0.4) to (A1) node {$\cic{1}$} to (0.7,-0.4) to (.7,-0.25);
\node (A) at (0,-1.2) {Only one possible cut of $D_rT^\fm(1,b+1)$, $b\ge 1$.};
\end{tikzpicture}
\end{center}

\begin{equation*}
\ncic{1}=\sum_{\eta_1,\eta_2=\pm 1} \eta_1\eta_2I^\fl(\eta_1;\eta_2,0_{r-1};0)\ot I^\fm(0;\eta_1,0_{b-r};1)=-T^\fl_{r-1}(1)\ot t^\fm(b-r+1)
\end{equation*}
which is unramified by Lemma~\ref{lem:singleTs}. Hence, $T^\fm(1,l)$ is unramified by Thm.~\ref{thm-Glanois}.

\medskip
\noindent
III. Suppose $a,b\ge 1$ and $r\le \min\{a,b\}.$ Clearly there are no nontrivial cuts for $D_r T^\fm(k,l)$.

\medskip
\noindent
IV. Suppose $a,b\ge 1$ and $r=a+s$ is odd, where $1\le s\le b.$
Then all possible cuts are shown in the following picture.

\begin{center}
\begin{tikzpicture}[scale=0.9]
\node (A0) at (0.05,0) {$0;\eta_1,0_{b-s},0,0_{a+s-b-1},\eta_2,0_{s-1},0,0_{b-s};1$};
\node (A1) at (-1.4,-0.4) {${}$};
\node (A2) at (-0.4,0.4) {${}$};
\node (A3) at (1.1,0.6) {${}$};
\draw (-3.6,-0.25) to (-3.6,-0.4) to (A1) node {$\cic{1}$} to (0.7,-0.4) to (.7,-0.25);
\draw (-3.2,0.25) to (-3.2,0.4) to (A2) node {$\cic{2}$} to (2.3,0.4) to (2.3,0.25);
\draw (-1.6,0.25) to (-1.6,0.6) to (A3) node {$\cic{3}$} to (3.75,0.6) to (3.75,0.25);
\node (A) at (0,-1.2) {Possible cuts of $D_{a+s}T^\fm(a+1,b+1)$, $a\ne 0$.};
\end{tikzpicture}
\end{center}
Here $\ncic{1}\,${} appears if $s=1$ and $a$ is even while $\ncic{3}\,${} appears if $a+s>b$. In this case
\begin{align*}
T(\ncic{1}\,)=&\, \sum_{\eta_1,\eta_2=\pm 1} \eta_1\eta_2I^\fl(0;\eta_1,0_a;\eta_2)\ot I^\fm(0;\eta_2,0_b;1)=T^\fl(a+1)\ot S^\fm(b+1),\\
S(\ncic{1}\,)=&\, \sum_{\eta_1,\eta_2=\pm 1} \eta_2I^\fl(0;\eta_1,0_a;\eta_2)\ot I^\fm(0;\eta_2,0_b;1)=S^\fl(a+1)\ot t^\fm(b+1).
\end{align*}
For general $1\le s\le b$ we always have (after putting $c=b-s+1$)
\begin{align*}
T(\ncic{2}\,)=&\, \sum_{\eta_1,\eta_2=\pm 1} \eta_1\eta_2I^\fl(\eta_1;0_a,\eta_2,0_{s-1};0)\ot I^\fm(0;\eta_1,0_c;1)
=-T^\fl_{s-1}(a+1)\ot S^\fm(c+1),\\
T(\ncic{3}\,)=&\, \sum_{\eta_1,\eta_2=\pm 1} \eta_1\eta_2I^\fl(0;0_{a+s-b-1},\eta_2,0_b;0)\ot I^\fm(0;\eta_1,0_c;1)
=T^\fl_{a-c}(b+1)\ot T^\fm(c+1),\\
S(\ncic{2}\,)=&\, \sum_{\eta_1,\eta_2=\pm 1} \eta_2 I^\fl(\eta_1;0_a,\eta_2,0_{s-1};0)\ot I^\fm(0;\eta_1,0_c;1)
=-T^\fl_{s-1}(a+1)\ot T^\fm(c+1),\\
S(\ncic{3}\,)=&\, \sum_{\eta_1,\eta_2=\pm 1} \eta_2I^\fl(0;0_{a+s-b-1},\eta_2,0_b;0)\ot I^\fm(0;\eta_1,0_c;1)
=T^\fl_{a-c}(b+1)\ot S^\fm(c+1).
\end{align*}
All of the above are clearly unramified by Lemma~\ref{lem:singleTs}. By Thm.~\ref{thm-Glanois} both $T^\fm(k,l)$ and $S^\fm(k,l)$ are unramified if $k,l\ge 2$. This completes the proof of the proposition.
\end{proof}

\begin{prop}\label{prop:DTwtEven}
Let $k,l\in\N$ have the same parity. Then the motivic double $T$-values $T^\fm(k,l)$ and double $S$-values $S^\fm(k,l)$ are ramified.
\end{prop}

\begin{proof}
Suppose $a=k-1,b=l-1$ have the same parity. Let $w=k+l$. If $b=0$ and $a=0$ then
\begin{align}
D_1 T^\fm(1,1)=&\, \sum_{\eta_1,\eta_2=\pm1}\eta_1\eta_2 \Big( I^\fl(0;\eta_1;\eta_2)\ot I^\fm(0;\eta_2;1)
+ I^\fl(\eta_1;\eta_2;1)\ot I^\fm(0;\eta_1;1) \Big) \notag \\
=&\, \sum_{\eta_1,\eta_2=\pm1}\eta_1\eta_2
 I^\fl(0;\eta_2;1)\ot I^\fm(0;\eta_1;1)
=\log^\fl 2\ot \log^\fm 2       \label{equ:D1T11}
\end{align}
Thus $T^\fm(1,1)$ is ramified by Thm.~\ref{thm-Glanois}. Similarly,
\begin{equation}\label{equ:D1S11}
D_1 S^\fm(1,1) = -\log^\fl 2\ot \log^\fm 2.
\end{equation}
Hence, $S^\fm(1,1)$ is ramified

If $b=0$ and $a>0$ is even then
\begin{align*}
D_1 T^\fm(a+1,1)=&\, \sum_{\eta_1,\eta_2=\pm1}\eta_1\eta_2 I^\fl(0;\eta_2;1)\ot I^\fm(0;\eta_1,0_a;1)
= \log^\fl 2\ot T^\fm(a+1), \\
D_{w-1} S^\fm(a+1,1)=&\, \sum_{\eta_1,\eta_2=\pm1} \eta_2 \Big(I^\fl(0;\eta_1,0_a;\eta_2)\ot I^\fm(0;\eta_2;1)
+I^\fl(\eta_1;0_a,\eta_2;1)\ot I^\fm(0;\eta_1;1) \Big) \\
= S^\fl(a+1) &\, \ot \log^\fm 2-T_a^\fl(1) \ot \log^\fm 2+ T^\fl(a+1) \ot \log^\fm 2=S^\fl(a+1) \ot \log^\fm 2
\end{align*}
which implies that both $T^\fm(k,1)$ and $S^\fm(k,1)$ are ramified by Thm.~\ref{thm-Glanois}.

If $b>0$ then
\begin{align}
&\, D_{w-1}T^\fm(a+1,b+1)= \sum_{\eta_1,\eta_2=\pm1}\eta_1\eta_2 D_{w-1} I^\fm(0;\eta_1,0_a,\eta_2,0_b;1) \notag \\
=&\, \sum_{\eta_1,\eta_2=\pm1}\eta_1\eta_2 I^\fl(\eta_1;0_a,\eta_2,0_b;1)\ot I^\fm(0;\eta_1;1) \notag \\
=&\, \sum_{\eta_2=\pm1} \eta_1\eta_2 \Big(I^\fl(0;0_a,\eta_2,0_b;1)-I^\fl(0;0_b,\eta_2,0_a;\eta_1)\Big)\ot I^\fm(0;\eta_1;1) \notag \\
=&\, \Big(T_a^\fl(b+1) + T_b^\fl(a+1) \Big)\ot \log^\fm 2 \notag \\
=&\, \big((-1)^a+(-1)^b\big)\binom{a+b}{a}T^\fl(w-1) \ot \log^\fm 2  \notag \\
=&\, (-1)^a 2\binom{a+b}{a}T^\fl(w-1) \ot \log^\fm 2 \not\in \calL_{w-1}\ot \calH_1 \label{equ:Dw-1DblTwtEven}
\end{align}
by Lemma~\ref{lem:singleTs},
which implies that $T^\fm(k,l)$ is ramified by Thm.~\ref{thm-Glanois}. Similarly, if $b>0$ then
\begin{equation}\label{equ:Dw-1DblSwtEven}
-D_{w-1}S^\fm(a+1,b+1)=D_{w-1}T^\fm(a+1,b+1)=(-1)^a 2\binom{a+b}{a}T^\fl(w-1)\ot \log^\fm 2
\end{equation}
is ramified and therefore $S^\fm(k,l)$ is ramified.
\end{proof}

\section{Unramified triple $S$-values}\label{sec:TripleSunramified}
In this section, we show that the conditions in Thm.~\ref{thm:MSV}(3) are sufficient, which is rephrased in the next proposition.

\begin{prop}\label{prop:MSVunr}
If $\bfk=(k,l,m)=(O_{>1},O,E)$, or $(O_{>1},E,O_{>1})$, or $(E,1,E)$ then $S^\fm(\bfk)$ is unramified.
\end{prop}

\begin{proof}
First, it follows immediately from Lemma \ref{lem:D1S1bfl} that $D_1 S^\fm(\bfk)=0$.
By Thm.~\ref{thm-Glanois} we now only need to show that the images of the motivic MSVs in the proposition
under $D_r$ are all unramified for odd $r$ such that $1<r<|\bfk|$.

If $r>1$ is odd then the same argument as in the proof of \cite[Thm.~41]{Murakami2021}
is mostly still valid and therefore $D_r S^\fm(\bfk)$ is unramified, at least when no component is 1.
However, there are some subtle differences when $1$ appears. We now consider three cases.

\medskip\noindent
(I) $1<r<k+l+m-1$.

\medskip
(I.1) When the weight is even, then both left and right factor of $D_rS^\fm(\bfk)$ are of odd weight $>1$ and depth at most two.
Thus both factors are unramified by Lemma~\ref{lem:singleTs}
and Prop.~\ref{prop:DBESreduce} if no component in the factors is 1.
If the left factor of $D_r$ has a component 1 (this can occur only when $\bfk=(k,l,m)=(O_{>1},1,E)$)
then it must have depth two and therefore the right factor has depth one and consequently the cut is unramified.
If the right factor has a component 1 and depth two, then the cut must straddle over exactly one of the $\eta$'s and
it must has the following form:
\begin{enumerate}
 \item [\upshape{(i)}] $I^\fl(0;0_s,\eta_i,0_{a_i};\eta_{i+1})\ot I^\fm(0;\dots,\eta_{i-1},0_{a_{i-1}-s},\eta_{i+1},\dots;1)
 \Longrightarrow k_\ell=1$ for some $\ell\ne i-1,i$.

\item [\upshape{(ii)}] $I^\fl(\eta_{i-1};0_{a_{i-1}-1},\eta_i,0_s;0)\ot I^\fm(0;\dots,\eta_{i-1},0_{a_i-s},\eta_{i+1},\dots;1)\Longrightarrow k_\ell=1$ for some $\ell\ne i-1,i$.

 \item [\upshape{(iii)}] $I^\fl(\eta_i;0_{a_i},\eta_{i+1},0_{a_{i+1}};\eta_{i+2})\ot I^\fm(0;\dots,\eta_i,\eta_{i+1},\dots;1) $.
\end{enumerate}
Only $\bfk=(O_{>1},O,E)$ need to be considered which involves both (i) and (iii). Let $k=a+1,m=c+1$. Then
\begin{align*}
&\, D_k S^\fm(k,1,m)\\
= &\, \sum_{\eta_1,\eta_2,\eta_3=\pm1} \eta_2\eta_3\Big( I^\fl(0;\eta_1,0_a;\eta_2)\ot I^\fm(0;\eta_2,\eta_3,0_c;1)
+ I^\fl(\eta_1;0_a,\eta_2;\eta_3)\ot I^\fm(0;\eta_1,\eta_3,0_c;1) \Big)\\
 = &\, \sum_{\eta_1,\eta_2,\eta_3=\pm1} \Big(\eta_2\eta_3 I^\fl(0;\eta_1,0_a;1)\ot I^\fm(0;\eta_2,\eta_3,0_c;1) \\
 &\, + \eta_1\eta_2\eta_3 I^\fl(1;0_a,\eta_2;0)\ot I^\fm(0;\eta_1,\eta_3,0_c;1)
+ \eta_2 I^\fl(0;0_a,\eta_2;1)\ot I^\fm(0;\eta_1,\eta_3,0_c;1) \Big)\\
 =&\, \Big(S^\fl(a+1)-T^\fl(a+1)\Big) \ot T^\fm(1,c+1) - T^\fl_a(1)\ot \tz^\fm(1,a+1)
\end{align*}
which is unramified by Lemma~\ref{lem:DTwtOdd} and Thm.~\ref{thm-distribution}. Thus $D_r S^\fm(\bfk)$ is unramified if $\bfk=(k,l,m)=(O_{>1},O,E)$ or $(O_{>1},E,O_{>1})$ with even weight and $1<r<k+l+m-1$ with $r$ odd.

\medskip
(I.2) If the weight is odd then $\bfk=(k,1,m)=(E,1,E)$. The computation is almost exactly the same as that in the proof of \cite[Lemma~41]{Murakami2021} except that we replace the factor $\eta_1\eta_2\eta_3$ by $\eta_2\eta_3$ inside all the sums for $D_r T^\fm(k_1,1,k_3)=D_r T^\fm(k,1,m)$ there (the overall sign difference is because our choices of definition of the motivic iterated integral have such a difference). Thus, among the five terms on the antepenultimate page of published version of \cite{Murakami2021}, the first, fourth and fifth terms are all unramified because the second factors are all of depth 1 and weight at least 2. The second and the third terms have the form
\begin{equation*}
\delta_{r<k} T^\fl(r)\ot \tz^\fm(k-r+1,m), \qquad \delta_{r<m} T^\fl_{r-1}(1)\ot \tz^\fm(k,m-r+1)
\end{equation*}
both of which are unramified since $k-r+1,m,k,m-r+1$ are all $\ge 2$.

\medskip\noindent
(II) $r=k+l+m-1$. Then we have
\begin{align*}
D_r S^\fm(k,l,m)
=&\, \sum_{\eta_1,\eta_2,\eta_3=\pm1} \eta_2\eta_3 I^\fl(\eta_1;0_a,\eta_2,0_b,\eta_3,0_c;1)\ot I^\fm(0;\eta_1;1) \\
=&\, \sum_{\eta_2,\eta_3=\pm1} \eta_2\eta_3 I^\fl(\bar 1;0_a,\eta_2,0_b,\eta_3,0_c;1) \ot \log^\fm 2 =0
\end{align*}
by \cite[Lemma~9]{Murakami2021}. Consequently, $S^\fm(k,l,m)$ is unramified by Thm.~\ref{thm-Glanois}.
\end{proof}

\section{Motivic triple $T$- and $S$-values: three special cases}\label{sec:TripleTspecial}
We consider three special cases in this section which have either a leading 1 or an ending 1 in the motivic triple $T$-values.

First, note that $T(k,l,1)$ all diverge and among the convergent triple $T$-values of the form $T(1,k,l)$
only $T(1,1,2)=T(4)$, $T(1,1,3)=T(1,4)$ and $T(1,2,2)=T(2,3)$ (by duality)
are unramified by the Kaneko-Tsumura Conjecture. It is straightforward to show that
$T^\fm(1,1,2)$, $T^\fm(1,1,3)$ and $T^\fm(1,2,2)$ are unramified by considering $D_1$ and $D_3$. This takes care of
Thm.~\ref{thm:MTV}(3)(i).

We now deal with all the other special cases by the next three propositions.

\begin{prop}\label{prop:TriTend1}
For any integers $k,l\in\N$, the motivic triple $S$- and $T$-values $S^\fm(k,l,1)$ and $T^\fm(k,l,1)$ are both ramified.
\end{prop}

\begin{proof} Let $a=k-1$ and $b=l-1$. If $b=0$ then we have
\begin{align*}
&\, D_{1}T^\fm(k,1,1)
= \sum_{\eta_1,\eta_2,\eta_3=\pm1}\eta_1\eta_2\eta_3 D_{1} I^\fm(0;\eta_1,0_a,\eta_2,\eta_3;1) \\
=&\, \sum_{\eta_1,\eta_2,\eta_3=\pm1}\eta_1\eta_2\eta_3 \Big(I^\fl(0;\eta_2;\eta_3)\ot I^\fm(0;\eta_1,0_a,\eta_3;1)
 +I^\fl(\eta_2;\eta_3;1)\ot I^\fm(0;\eta_1,0_a,\eta_2;1)\Big)\\
=&\, \sum_{\eta_1,\eta_3=\pm1} \eta_1 \log^\fl 2\ot I^\fm(0;\eta_1,0_a,\eta_3;1)
 +\sum_{\eta_1,\eta_3=\pm1} I^\fl(\bar1;\eta_3;1)\ot I^\fm(0;\eta_1,0_a,\bar1;1)\\
=&\, \log^\fl 2 \ot t^\fm(a+1,1) \ne 0
\end{align*}
since $\sum_{\eta_3=\pm1} I^\fl(\bar1;\eta_3;1)=0$ by Lemma~\ref{Murakami-lemma9}. Similarly,
\begin{equation*}
D_{1}S^\fm(k,1,1)= \log^\fl 2 \ot \tz^\fm(a+1,1)\ne 0.
\end{equation*}
Thus both $S^\fm(k,1,1)$ and $T^\fm(k,1,1)$ are ramified by Thm.~\ref{thm-Glanois}.

If $b\ge 1$ then we have only one cut for $D_1$:
\begin{align*}
 D_{1}T^\fm(k,l,1)
=&\, \sum_{\eta_1,\eta_2,\eta_3=\pm1}\eta_1\eta_2\eta_3 D_{1} I^\fm(0;\eta_1,0_a,\eta_2,0_b,\eta_3;1) \\
=&\, \sum_{\eta_1,\eta_2,\eta_3=\pm1}\eta_1\eta_2\eta_3 I^\fl(0;\eta_3;1)\ot I^\fm(0;\eta_1,0_a,\eta_2,0_b;1) \\
=&\, \log^\fl 2 \ot T^\fm(a+1,b+1) \ne 0.
\end{align*}
Similarly,
\begin{equation*}
D_{1}S^\fm(k,l,1)= \log^\fl 2 \ot S^\fm(a+1,b+1) \ne 0.
\end{equation*}
Hence, both $S^\fm(k,l,1)$ and $T^\fm(k,l,1)$ are ramified by Thm.~\ref{thm-Glanois}. This completes the proof of the proposition.
\end{proof}

\begin{prop}\label{prop:TriTstart11}
For all integers $k\ge 4$ the motivic triple $T$-values $T^\fm(1,1,k)$ are ramified.
\end{prop}

\begin{proof}
Write $a=k-1\ge 3$. Then
\begin{align*}
&\, D_{3}T^\fm(1,1,a+1)
= \sum_{\eta_1,\eta_2,\eta_3=\pm1}\eta_1\eta_2\eta_3 D_{3} I^\fm(0;\eta_1,\eta_2,\eta_3,0_a;1) \\
=&\, \sum_{\eta_1,\eta_2,\eta_3=\pm1}\eta_1\eta_2\eta_3 \Big(I^\fl(\eta_1;\eta_2,\eta_3,0;0)\ot I^\fm(0;\eta_1,0_{a-1};1)
 +I^\fl(\eta_2;\eta_3,0,0;0)\ot I^\fm(0;\eta_1,\eta_2,0_{a-2};1)\Big)\\
=&\, -\sum_{\eta_1,\eta_2,\eta_3=\pm1}\eta_1\eta_3 \Big(\eta_2 I^\fl(0;0,\eta_3,\eta_2;1)\ot I^\fm(0;\eta_1,0_{a-1};1)
 +I^\fl(0;0,0,\eta_3;1)\ot I^\fm(0;\eta_1,\eta_2,0_{a-2};1)\Big)\\
=&\, -T^\fl_1(1,1) \ot T^\fm(a)-T^\fl_2(1) \ot t^\fm(1,a-1)\\
=&\, \Big(T^\fl(2,1)+T^\fl(1,2)\Big) \ot (2-2^{1-a})\zeta^\fm(a)-T^\fl(3) \ot t^\fm(1,a-1)\\
=&\, -\frac74(2-2^{1-a})\zeta^\fl(3) \ot\zeta^\fm(a)-\frac74\zeta^\fl(3) \ot t^\fm(1,a-1)
\end{align*}
by Lemma~\ref{lem:singleTs} and Cor.~\ref{cor:DBreduce}.
In the above and throughout the rest of the paper, we often use substitutions
such as $\eta_3\to \eta_3\eta_1$ to convert the motivic integrals to the standard forms of motivic multiple $T$-, $t$- and $S$-values.
Since $t^\fm(1,a-1)$ is ramified by Prop.~\ref{prop:MtV1s} $D_{3}T^\fm(1,1,a+1)\not\in \calL_3\ot\calH_a$
and therefore $T^\fm(1,1,a+1)$ is ramified by Thm.~\ref{thm-Glanois}.
\end{proof}

\begin{prop}\label{prop:TriSstart11}
For all integers $k\in\N$ the motivic triple $S$-values $S^\fm(1,1,k)$ are ramified.
\end{prop}

\begin{proof}
Write $a=k-1$. First, we assume $a\ge 1$. Then
\begin{align*}
&\, D_1 S^\fm(1,1,a+1)
= \sum_{\eta_1,\eta_2,\eta_3=\pm1} \eta_2\eta_3 D_1 I^\fm(0;\eta_1,\eta_2,\eta_3,0_a;1) \\
=&\, \sum_{\eta_1,\eta_2,\eta_3=\pm1}\eta_2\eta_3 \Big(I^\fl(0;\eta_1;\eta_2)\ot I^\fm(0;\eta_2,\eta_3,0_a;1) \\
&\, +I^\fl(\eta_1;\eta_2;\eta_3)\ot I^\fm(0;\eta_1,\eta_3,0_a;1)
 +I^\fl(\eta_2;\eta_3;0)\ot I^\fm(0;\eta_1,\eta_2,0_a;1)\Big)\\
=&\, \sum_{\eta_1,\eta_2,\eta_3=\pm1}\eta_2\eta_3 \Big(I^\fl(0;\eta_1;\eta_2)\ot I^\fm(0;\eta_2,\eta_3,0_a;1) -I^\fl(0;\eta_2;\eta_3)\ot I^\fm(0;\eta_1,\eta_3,0_a;1) \Big)\\
=&\, \sum_{\eta_1,\eta_2,\eta_3=\pm1}(\eta_2-\eta_1)\eta_3 I^\fl(0;\eta_1;\eta_2)\ot I^\fm(0;\eta_2,\eta_3,0_a;1) \\
=&\, -2\log^\fl 2 \ot T^\fm(1,a+1) \ne 0.
\end{align*}
Hence, $S^\fm(1,1,k)$ is ramified by Thm.~\ref{thm-Glanois}. If $a=0$ then
\begin{align*}
&\, D_1 S^\fm(1,1,1)
= \sum_{\eta_1,\eta_2,\eta_3=\pm1} \eta_2\eta_3 D_1 I^\fm(0;\eta_1,\eta_2,\eta_3;1) \\
=&\, \sum_{\eta_1,\eta_2,\eta_3=\pm1}\eta_2\eta_3 \Big(I^\fl(0;\eta_1;\eta_2)\ot I^\fm(0;\eta_2,\eta_3;1) \\
&\, +I^\fl(\eta_1;\eta_2;\eta_3)\ot I^\fm(0;\eta_1,\eta_3,0_a;1)
 +I^\fl(\eta_2;\eta_3;1)\ot I^\fm(0;\eta_1,\eta_2;1)\Big)\\
=&\, \sum_{\eta_1,\eta_2,\eta_3=\pm1}\Big((\eta_2-\eta_1)\eta_3 I^\fl(0;\eta_1;\eta_2)\ot I^\fm(0;\eta_2,\eta_3;1)
+ \eta_2\eta_3 I^\fl(0;\eta_3;1)\ot I^\fm(0;\eta_1,\eta_2;1) \Big) \\
=&\, -\log^\fl 2 \ot \Big(2T^\fm(1,1)-S^\fm(1,1)\Big) \ne 0
\end{align*}
since $-D_1 T^\fm(1,1)=D_1 S^\fm(1,1)=-\log^\fl 2 \ot \log^\fm 2\ne 0$ by \eqref{equ:D1T11} and \eqref{equ:D1S11}.
Hence, $S^\fm(1,1,k)$ is ramified by Thm.~\ref{thm-Glanois}.
\end{proof}

\begin{prop}\label{prop:TriTstart1}
For all integers $k,l\ge 2$ the motivic triple $T$- and $S$-values $T^\fm(1,k,l)$ (if $(k,l)\ne(2,2)$)
and $S^\fm(1,k,l)$ are ramified.
\end{prop}

\begin{proof}
Let $a=k-1$ and $b=l-1$. We break the proposition into three cases:
\begin{center}
I. $a$ and $b$ are even; \quad II. $a$ is even, $b$ is odd; \quad III. $a$ is odd.
\end{center}

\medskip
\noindent
\textbf{Case I}. Suppose $a,b\ge 2$ and $a$ and $b$ are both even.
Since $a+b\ge 4$ then we have the following possible cuts for $D_{a+1}$.

\begin{center}
\begin{tikzpicture}[scale=0.9]
\node (A0) at (0.05,0) {$0;\eta_1,\eta_2,0_{b-1},0,0_{a-b},\eta_3,0,0_{b-1};1$};
\node (A1) at (-0.8,-0.4) {${}$};
\node (A2) at (-0.6,0.4) {${}$};
\node (A3) at (1.3,0.6) {${}$};
\draw (-2.7,-0.25) to (-2.7,-0.4) to (A1) node {$\cic{1}$} to (1.1,-0.4) to (1.1,-0.25);
\draw (-2.0,0.25) to (-2.0,0.4) to (A2) node {$\cic{2}$} to (1.65,0.4) to (1.65,0.25);
\draw (-.4,0.25) to (-.4,0.6) to (A3) node {$\cic{3}$} to (3.2,0.6) to (3.2,0.25);
\node (A) at (0,-1.2) {Possible cuts of $D_{a+1}T^\fm(1,a+1,b+1)$, $a$ even and $b\ge 2$.};
\end{tikzpicture}
\end{center}

First, since $a$ is even we observe that
$$
I^\fm(\eta_1;\eta_2,0_a;0)=-\zeta^\fm_a(1;\eta_1\eta_2)=-\zeta^\fm(a+1;\eta_1\eta_2)=-I^\fm(0;\eta_2,0_a;\eta_1).
$$
We get
\begin{align*}
T(\ncic{1}\,)=&\, \sum_{\eta_1,\eta_2,\eta_3=\pm1}\eta_1\eta_2\eta_3 I^\fl(\eta_1;\eta_2,0_a;\eta_3) \ot I^\fm(0;\eta_1,\eta_3,0_b;1) \\
=&\, \sum_{\eta_1,\eta_2,\eta_3=\pm1}\eta_1\eta_2\eta_3 \Big(I^\fl(0;\eta_2,0_a;\eta_3)-I^\fl(0;0_a,\eta_2;\eta_1) \Big) \ot I^\fm(0;\eta_1,\eta_3,0_b;1) \\
=&\, \sum_{\eta_1,\eta_2,\eta_3=\pm1} \Big(\eta_1\eta_2I^\fl(0;\eta_2,0_a;1)-\eta_2\eta_3I^\fl(0;0_a,\eta_2;1) \Big) \ot I^\fm(0;\eta_1,\eta_3,0_b;1) \\
=&\, T^\fl(a+1) \ot \Big(t^\fm(1,b+1)-S^\fm(1,b+1)\Big), \\
T(\ncic{2}\,)=&\, \sum_{\eta_1,\eta_2,\eta_3=\pm1}\eta_1\eta_2\eta_3 I^\fl(\eta_2;0_a,\eta_3;0) \ot I^\fm(0;\eta_1,\eta_2,0_b;1)
=- T^\fl(a+1) \ot t^\fm(1,b+1).
\end{align*}
Since $b$ is even, by Lemma~\ref{lem:singleTs} and Prop.~\ref{prop:DTwtEven} we see that
\begin{equation*}
 T( \ncic{1}+\ncic{2}\,)=
\left\{
 \begin{array}{ll}
 - T^\fl(a+1) \ot S^\fm(1,b+1) \in \calL_{a+1} \ot (\calH_{b+1}^2 \setminus \calH_{b+1}) , \phantom{\frac12}& \hbox{if $a>0$;} \\
 - \log^\fl 2 \ot S^\fm(1,b+1) \in (\calL_1^2 \setminus\calL_1) \ot (\calH_{b+1}^2 \setminus \calH_{b+1}) ,\phantom{\frac12} & \hbox{if $a=0$.}
 \end{array}
\right.
\end{equation*}
If $a\ge b$ then we need to consider the additional term
\begin{align*}
T(\ncic{3}\,)= &\, \sum_{\eta_1,\eta_2,\eta_3=\pm1}\eta_1\eta_2\eta_3 I^\fl(0;0_{a-b},\eta_3,0_b;1) \ot I^\fm(0;\eta_1,\eta_2,0_{b};1) \\
=&\, T^\fl_{a-b}(b+1) \ot T^\fm(1,b+1)
=\binom{a}{b}T^\fl(a+1) \ot T^\fm(1,b+1)
\end{align*}
by Lemma~\ref{lem:singleTs} since $(-1)^{a-b}=1$. Thus
\begin{equation*}
 T(\ncic{1}+\ncic{2}+\ncic{3}\,)=
 T^\fl(a+1) \ot \left(\binom{a}{b}T^\fm(1,b+1)-S^\fm(1,b+1)\right).
\end{equation*}
If $a\ne b$, by \eqref{equ:Dw-1DblSwtEven} we have
\begin{equation*}
D_{b+1} \left(\binom{a}{b}T^\fm(1,b+1)-S^\fm(1,b+1)\right)=2\left(\binom{a}{b}+1\right)T^\fl(b+1) \ot \log^\fm 2
\end{equation*}
which is ramified. Thus $T^\fm(1,k,l)$ is ramified in this case.

For $S^\fm(1,k,l)$, if $a+b\ge 4$ then similarly we have by Prop.~\ref{prop:DTwtEven}
\begin{equation*}
 S(\ncic{1}+\ncic{2}\,)=
\left\{
 \begin{array}{ll}
 - T^\fl(a+1) \ot T^\fm(1,b+1) \in \calL_{a+1} \ot (\calH_{b+1}^2 \setminus \calH_{b+1}) , \phantom{\frac12}& \hbox{if $a>0$;} \\
 - \log^\fl 2 \ot T^\fm(1,b+1) \in (\calL_1^2 \setminus\calL_1) \ot (\calH_{b+1}^2 \setminus \calH_{b+1}) ,\phantom{\frac12} & \hbox{if $a=0$.}
 \end{array}
\right.
\end{equation*}
If $a\ge b$ then we need to consider the additional term
\begin{align*}
S(\ncic{3}\,)=\binom{a}{b}T^\fl(a+1) \ot S^\fm(1,b+1).
\end{align*}
Thus
\begin{equation*}
 S(\ncic{1}+\ncic{2}+\ncic{3}\,)=
 T^\fl(a+1) \ot \left(\binom{a}{b}S^\fm(1,b+1)+T^\fm(1,b+1)\right).
\end{equation*}
Similar to the above we see this is ramified and therefore so is $S^\fm(1,k,l)$.

\medskip
\noindent
\textbf{Case II}. Suppose $a,b\ge 1$, $a$ is even and $b$ is odd.

(i) Assume $a>b$. Then
\begin{align*}
&\,  D_{a+1}S^\fm(1,a+1,b+1)\\
=&\, \sum_{\eta_1,\eta_2,\eta_3=\pm1} \eta_2\eta_3 I^\fl(\eta_1;\eta_2,0_a;\eta_3) \ot I^\fm(0;\eta_1,\eta_3,0_b;1)\\
&\, +\sum_{\eta_1,\eta_2,\eta_3=\pm1} \eta_2\eta_3 I^\fl(\eta_2;0_a,\eta_3;0) \ot I^\fm(0;\eta_1,\eta_2,0_b;1)\\
&\, +\sum_{\eta_1,\eta_2,\eta_3=\pm1} \eta_2\eta_3 I^\fl(0;0_{a-b},\eta_3,0_b;1) \ot I^\fm(0;\eta_1,\eta_2,0_b;1)\\
=&\, \sum_{\eta_1,\eta_2,\eta_3=\pm1} \eta_2\eta_3 \big(I^\fl(\eta_1;\eta_2,0_a;0)+I^\fl(0;\eta_2,0_a;\eta_3)\big) \ot  I^\fm(0;\eta_1,\eta_3,0_b;1)\\
&\, -\sum_{\eta_1,\eta_2,\eta_3=\pm1}\eta_2\eta_3 I^\fl(0;\eta_3,0_a;\eta_2) \ot I^\fm(0;\eta_1,\eta_2,0_b;1)+T^\fl_{a-b}(b+1)\ot S^\fm(1,b+1)\\
=&\, -\sum_{\eta_1,\eta_2,\eta_3=\pm1} \eta_2\eta_3 I^\fl(0;0_a,\eta_2;\eta_1)\ot  I^\fm(0;\eta_1,\eta_3,0_b;1)  +T^\fl_{a-b}(b+1)\ot S^\fm(1,b+1)\\
=&\, -T^\fl_{a}(1)\ot T^\fm(1,b+1) +T^\fl_{a-b}(b+1)\ot S^\fm(1,b+1) \\
=&\, -T^\fl(a+1)\ot\bigg( T^\fm(1,b+1)+\binom{a+b}{2b} S^\fm(1,b+1)\bigg).
\end{align*}
Similar computations shows that
\begin{align*}
 D_{a+1}T^\fm(1,a+1,b+1)=&\, -T^\fl(a+1)\ot\bigg(S^\fm(1,b+1)+\binom{a+b}{2b} T^\fm(1,b+1)\bigg).
\end{align*}
By Prop.~\ref{prop:DTwtOdd} $T^\fm(1,b+1)$ is unramified while $S^\fm(1,b+1)$ is ramified by Lemma~\ref{lem:D1S1bfl} we see that
both $S^\fm(1,a+1,b+1)$ and $T^\fm(1,a+1,b+1)$ are ramified in this case.

(ii)  Assume $b>a$. Then
\begin{align*}
&\,  D_{b+2}S^\fm(1,a+1,b+1)\\
=&\, \sum_{\eta_1,\eta_2,\eta_3=\pm1} \eta_2\eta_3 I^\fl(\eta_1;\eta_2,0_a,\eta_3,0_{b-a};0) \ot I^\fm(0;\eta_1,0_a;1)\\
&\, +\sum_{\eta_1,\eta_2,\eta_3=\pm1} \eta_2\eta_3 I^\fl(\eta_2;0_a,\eta_3,0_{b-a+1};0) \ot I^\fm(0;\eta_1,\eta_2,0_{a-1};1) \\
&\, +\sum_{\eta_1,\eta_2,\eta_3=\pm1} \eta_2\eta_3 I^\fl(0;0,\eta_3,0_{b};1) \ot I^\fm(0;\eta_1,\eta_2,0_{a-1};1) \\
=&\,  - T^\fl_{b-a}(1,a+1)\ot \tz^\fm(a+1) -T^\fl_{b-a+1}(a+1)\ot \tz^\fm(1,a)+T^\fl_{1}(b+1)\ot S^\fm(1,a) \\
=&\,   T^\fl(b+2)\ot \Big[c_1\tz^\fm(a+1)-\binom{b+1}{a}\tz^\fm(1,a)-(b+1) S^\fm(1,a) \Big]
\end{align*}
by (I3) and Cor.~\ref{cor:DBreduce}, where $c_1$ and $c_2$ are some constants. Since both $\tz^\fm(a+1)$ and $\tz^\fm(1,a+1)$ are unramified
while $S^\fm(1,a)$ is ramified by Lemma~\ref{lem:D1S1bfl}, we see that $S^\fm(1,a+1,b+1)$ is ramified in this case.

Similar computations shows that
\begin{align*}
 D_{b+2}T^\fm(1,a+1,b+1)=&\,  T^\fl(b+2)\ot \Big[c_1 T^\fm(a+1)-\binom{b+1}{a} t^\fm(1,a)-(b+1) T^\fm(1,a) \Big].
\end{align*}
Hence $T^\fm(1,a+1,b+1)$ is ramified since both $T^\fm(a+1)$ and $T^\fm(1,a)$ are unramified by Prop.~\ref{prop:DTwtOdd} while $t^\fm(1,a)$ is ramified by Lemma~\ref{lem:Charlton}.

\medskip
\noindent
\textbf{Case III}.
Suppose $a,b\ge 1$ and $a$ is odd. It is easy to see that
\begin{align*}
D_1 S^\fm(1,a+1,b+1)=&\, \sum_{\eta_1,\eta_2,\eta_3=\pm1} \eta_2\eta_3
\Big(I^\fl(0;\eta_1;\eta_2) \ot I^\fm(0;\eta_2,0_a,\eta_3,0_b;1) \\
&\, \qquad +I^\fl(\eta_1;\eta_2;0) \ot I^\fm(0;\eta_1,0_a,\eta_3,0_b;1) \Big)\\
=&\, \sum_{\eta_1,\eta_2,\eta_3=\pm1}(\eta_2-\eta_1)\eta_3 I^\fl(0;\eta_1;\eta_2) \ot I^\fm(0;\eta_2,0_a,\eta_3,0_b;1) \\
=&\, -\log^\fl 2\ot T^\fm(a+1,b+1)
\end{align*}
which is ramified. Thus $S^\fm(1,a+1,b+1)$ is ramified.

\bigskip
Now we turn to $T^\fm(1,a+1,b+1)$. We divide into two cases: (A) $b=1$ (B) $b\ge 2$.

\medskip
\noindent
(A) If $b=1$ then $a\ge 3$ and
\begin{align*}
D_1 T^\fm(1,a+1,2) = \sum_{\eta_1,\eta_2,\eta_3=\pm1} \eta_1 \eta_2\eta_3 I^\fl(0;\eta_3;1) \ot I^\fm(0;\eta_1,\eta_2,0_a;1)
=\log^\fl 2\ot T^\fm(1,a+1)
\end{align*}
yielding that $T^\fm(1,a+1,2)$ is ramified in this case.

\medskip
\noindent
(B) Assume $b\ge 2$. We have the following possible cuts for $D_{a+2}T^\fm(1,a+1,b+1)$.

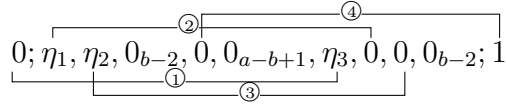
\begin{center}
\begin{tikzpicture}[scale=0.9]
\node (A0) at (0.05,0) {$0;\eta_1,\eta_2,0_{b-2},0,0_{a-b+1},\eta_3,0,0,0_{b-2};1$};
\node (A1) at (-1.2,-0.4) {${}$};
\node (A2) at (-1.0,0.4) {${}$};
\node (A3) at (-0.1,-0.6) {${}$};
\node (A4) at (1.4,0.6) {${}$};
\draw (-3.6,-0.25) to (-3.6,-0.4) to (A1) node {$\cic{1}$} to (1.2,-0.4) to (1.2,-0.25);
\draw (-3.0,0.25) to (-3.0,0.4) to (A2) node {$\cic{2}$} to (1.7,0.4) to (1.7,0.25);
\draw (-2.4,-0.25) to (-2.4,-0.6) to (A3) node {$\cic{3}$} to (2.2,-0.6) to (2.2,-0.25);
\draw (-.8,0.25) to (-.8,0.6) to (A4) node {$\cic{4}$} to (3.6,0.6) to (3.6,0.25);
\node (A) at (0,-1.2) {Possible cuts of $D_{a+2}T^\fm(1,a+1,b+1)$, $a$ odd and $b\ge2$.};
\end{tikzpicture}
\end{center}

We get
\begin{align*}
\ncic{1}=&\, \sum_{\eta_1,\eta_2,\eta_3=\pm1}\eta_1\eta_2\eta_3 I^\fl(0;\eta_1,\eta_2,0_a;\eta_3) \ot I^\fm(0;\eta_3,0_b;1)
= T^\fl(1,a+1) \ot T^\fm(b+1) , \\
\ncic{2}=&\, \sum_{\eta_1,\eta_2,\eta_3=\pm1}\eta_1\eta_2\eta_3 I^\fl(\eta_1;\eta_2,0_a,\eta_3;0) \ot I^\fm(0;\eta_1,0_b;1) \\
=&\, -\sum_{\eta_1,\eta_2,\eta_3=\pm1}\eta_1\eta_2\eta_3 I^\fl(0;\eta_1,0_a,\eta_2;\eta_3) \ot I^\fm(0;\eta_3,0_b;1)
= - T^\fl(a+1,1) \ot T^\fm(b+1) .
\end{align*}
Since $a$ is odd so that both $T^\fl(1,a+1)$ and $T^\fl(a+1,1)$ have odd weight we see that
\begin{equation}\label{equ:1and2}
 \ncic{1}+\ncic{2}=\Big(T^\fl(1,a+1)-T^\fl(a+1,1)\Big) \ot T^\fm(b+1) \in \calL_{a+2}\ot \calH_{b+1}
\end{equation}
by Lemma~\ref{lem:DTwtOdd}. Further
\begin{align}
\ncic{3}=&\, \sum_{\eta_1,\eta_2,\eta_3=\pm1}\eta_1\eta_2\eta_3
 I^\fl(\eta_2;0_a,\eta_3,0;0) \ot I^\fm(0;\eta_1,\eta_2,0_{b-1};1) \notag\\
=&\,- \sum_{\eta_1,\eta_2,\eta_3=\pm1}\eta_1 \eta_3 I^\fl(0;0,\eta_2\eta_3,0_a;\eta_2) \ot I^\fm(0;\eta_1,\eta_2,0_{b-1};1) \notag \\
=&\, - T^\fl_1(a+1) \ot t^\fm(1,b) = (a+1)T^\fl(a+2)\ot t^\fm(1,b) \label{equ:DT1kl3}
\end{align}
by Lemma~\ref{lem:singleTs}, which is ramified by Prop.~\ref{prop:MtV1s}.

\smallskip
\noindent
\textbf{(1)} If $b\ge2$ and $a<b-1$ then $\ncic{4}\,${} does not appear and we see that
$D_{a+2}T^\fm(1,a+1,b+1)=\ncic{1}+\ncic{2}+\ncic{3}\,${} is ramified by \eqref{equ:1and2} and \eqref{equ:DT1kl3}.

\smallskip
\noindent
\textbf{(2)} If $a\ge b-1\ge1$ then $\ncic{4}\,${} appears and we get
\begin{align*}
\ncic{4}=&\, \sum_{\eta_1,\eta_2,\eta_3=\pm1}\eta_1\eta_2\eta_3 I^\fl(0;0_{a-b+1},\eta_3,0_b;1) \ot I^\fm(0;\eta_1,\eta_2,0_{b-1};1) \\
=&\, T^\fl_{a-b+1}(b+1) \ot T^\fm(1,b)
= (-1)^{b}\binom{a+1}{b}T^\fl(a+2) \ot T^\fm(1,b).
\end{align*}
By combining $\ncic{3}\,${} and $\ncic{4}\,${} we get
\begin{equation*}
\ncic{3}+\ncic{4}= T^\fl(a+2) \ot \bigg( (a+1) t^\fm(1,b)+(-1)^{b}\binom{a+1}{b} T^\fm(1,b)\bigg).
\end{equation*}
Since $D_1 T^\fm(1,b)=0$ we see that by Lemma~\ref{lem:Charlton}
\begin{equation*}
D_1\left((a+1) t^\fm(1,b)+ (-1)^{b}\binom{a+1}{b} T^\fm(1,b)\right)
=2 (a+1) \log^\fl 2\ot t^\fm(b)\ne 0.
\end{equation*}
Thus $\ncic{3}+\ncic{4}\in\calL_{a+2}\ot (\calH_{b+1}^2\setminus \calH_{b+1} )$.
Combining this with \eqref{equ:1and2} we find that $D_{a+2}T^\fm(1,a+1,b+1)\not\in \calL_{a+2}\ot \calH_{b+1}$.

Putting (A) and (B) together we can conclude that if $a+b\ge 3$ and $a$ is odd then
$T^\fm(1,a+1,b+1)$ is ramified by Thm.~\ref{thm-Glanois}.
\end{proof}

\section{Triple $T$- and $S$-values: the general case}\label{sec:TripleTandS}
We now consider motivic triple $T$-values $T^\fm(k,l,m)$ and motivic triple $S$-values $S^\fm(k,l,m)$. In view of Prop.~\ref{prop:TriTend1} and Prop.~\ref{prop:TriTstart1} we may assume that $a=k-1\ge 1,c=m-1\ge 1$ and $b=l-1\ge0$ such that $w=k+l+m$. Since $T^\fm(O_{>1},E,O_{>1})$ are all unramified by Murakami (see \cite[Thm.~41]{Murakami2021}),
$S^\fm(E,O_{>1},O_{>1})$ and $S^\fm(O_{>1},E,O_{>1})$ are all unramified by Prop.~\ref{prop:MSVunr},
we only consider the following three cases for triple $T^\fm$-values and two cases for triple $S^\fm$-values when $w$ is even:
\begin{center}
 $f^\fm(E,E,E), \quad f^\fm(E,O,O_{>1}), \quad T^\fm(O_{>1},O,E),$
\end{center}
where $f=T$ or $S$. When $w$ is odd, since
$T^\fm(E,1,E)$ are all unramified by Murakami (see \cite[Thm.~41]{Murakami2021})
and $S^\fm(E,1,E)$ are all unramified by Prop.~\ref{prop:MSVunr}
we only consider the following three cases for triple $T^\fm$-values and three cases for triple $S^\fm$-values
when $w$ is odd:
\begin{center}
 $f^\fm(E,E,O_{>1}), \quad f^\fm(O_{>1},E,E), \quad f^\fm(E,O_{>1},E)$,
\end{center}
where $f=T$ or $S$.

\medskip
\noindent
\textbf{Case I}. (This includes $f^\fm(E,E,E)$,  $f^\fm(E,E,O_{>1})$ and $T^\fm(O_{>1},O,E)$.)
Suppose $a$ and $b$ have the same parity.
Then we have the following possible cuts for $D_{a+b+1}T^\fm(k,l,m)$.

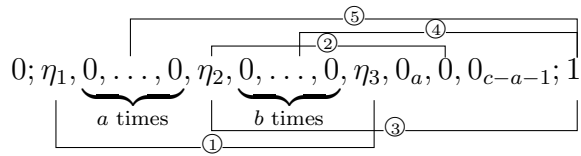
\begin{center}
\begin{tikzpicture}[scale=0.9]
\node (A0) at (0.05,0) {$0;\eta_1,\underbrace{0,\dots,0}_{\scriptstyle \text{$a$ times}},\eta_2,\underbrace{0,\dots,0}_{\scriptstyle \text{$b$ times}},\eta_3,0_a,0,0_{c-a-1};1$};
\node (A1) at (-1.2,-0.8) {${}$};
\node (A2) at (0.5,0.7) {${}$};
\node (A3) at (1.5,-0.55) {${}$};
\node (A4) at (2.1,.9) {${}$};
\node (A5) at (0.9,1.1) {${}$};
\draw (-3.5,-0.1) to (-3.5,-0.8) to (A1) node {$\cic{1}$} to (1.2,-0.8) to (1.2,-0.1);
\draw (-1.2,0.55) to (-1.2,0.7) to (A2) node {$\cic{2}$} to (2.25,0.7) to (2.25,0.53);
\draw (-1.2,-0.1) to (-1.2,-0.55) to (A3) node {$\cic{3}$} to (4.2,-0.55) to (4.2,-0.1);
\draw (0.1,0.55) to (0.1,.9) to (A4) node {$\cic{4}$} to (4.2,.9) to (4.2,0.53);
\draw (-2.4,0.55) to (-2.4,1.1) to (A5) node {$\cic{5}$} to (4.2,1.1) to (4.2,0.53);
\node (A) at (0,-1.3) {Possible cuts of $D_{a+b+1} f^\fm(k,l,m)$ ($f=T$ or $S$), $a\equiv b\pmod{2}$.};
\end{tikzpicture}
\end{center}

We have
\begin{align*}
D_{a+b+1}T^\fm(k,l,m)=&\, \sum_{\eta_1,\eta_2,\eta_3=\pm1}\eta_1\eta_2\eta_3 \bigg[I^\fl(\eta_1;0_a,\eta_2,0_b;\eta_3) \ot I^\fm(0;\eta_1,\eta_3,0_c;1) \\
&\, +\delta_{c>a+1} I^\fl(\eta_2;0_b,\eta_3,0_a;0) \ot I^\fm(0;\eta_1,0_a,\eta_2,0_{c-a};1) \\
&\, +\delta_{a=c} I^\fl(\eta_2;0_b,\eta_3,0_a;1) \ot I^\fm(0;\eta_1,0_a,\eta_2;1) \\
&\, +\delta_{c+1>a} I^\fl(0;0_{b-c+a},\eta_3,0_c;1) \ot I^\fm(0;\eta_1,0_a,\eta_2,0_{c-a};1) \\
&\, + \delta_{c+1<a} I^\fl(0;0_{a-c-1},\eta_2,0_b,\eta_3,0_c;1) \ot I^\fm(0;\eta_1,0_{c+1};1) \bigg].
\end{align*}
Since $a$ and $b$ have the same parity, we have
\begin{align*}
\ncic{1}=&\, \sum_{\eta_1,\eta_2,\eta_3=\pm1}\eta_1\eta_2\eta_3 \Big(I^\fl(0;0_a,\eta_2,0_b;\eta_3)+I^\fl(\eta_1;0_a,\eta_2,0_b;0)\Big)\ot I^\fm(0;\eta_1,\eta_3,0_c;1) \\
=&\, \sum_{\eta_1,\eta_2,\eta_3=\pm1} \Big(\eta_1\eta_2I^\fl(0;0_a,\eta_2,0_b;1)+\eta_2\eta_3I^\fl(1;0_a,\eta_2,0_b;0)\Big)\ot I^\fm(0;\eta_1,\eta_3,0_c;1) \\
=&\, T^\fl_a(b+1) \ot \Big( t^\fm(1,c+1)-S^\fm(1,c+1)\Big)\\
=&\,(-1)^a \binom{a+b}{a} T^\fl(a+b+1)\ot \Big( t^\fm(1,c+1)-S^\fm(1,c+1)\Big),\\
\ncic{2}=&\, -\sum_{\eta_1,\eta_2,\eta_3=\pm1} \eta_3 I^\fl(0;0_a,\eta_3,0_b;1) \ot \eta_1 I^\fm(0;\eta_1,0_a,\eta_2,0_{c-a};1) \\
=&\, - T^\fl_a(b+1) \ot t^\fm(a+1,c-a+1)= -(-1)^a \binom{a+b}{a} T^\fl(a+b+1)\ot t^\fm(a+1,c-a+1),\\
\ncic{3}=&\, \sum_{\eta_1,\eta_3=\pm1} \eta_1\eta_3 \Big(I^\fl(0;0_b,\eta_3,0_a;1)+\eta_2 I^\fl(1;0_b,\eta_3,0_a;0)\Big)\ot I^\fm(0;\eta_1,0_a,\eta_2;1) \\
=&\,T^\fl_b(a+1) \ot t^\fm(a+1,1) - T^\fl_a(b+1) \ot T^\fm(a+1,1)\\
= &\, -(-1)^a 2 \binom{a+b}{a} T^\fl(a+b+1)\ot \Big(t^\fm(a+1,1)-T^\fm(a+1,1)\Big),\\
\ncic{4}=&\, \sum_{\eta_1,\eta_2,\eta_3=\pm1}\eta_1\eta_2\eta_3 I^\fl(0;0_{b-c+a},\eta_3,0_c;1) \ot I^\fm(0;\eta_1,0_a,\eta_2,0_{c-a};1) \\
=&\, T^\fl_{b-c+a}(c+1)\ot T^\fm(a+1,c-a+1)
= (-1)^{c} \binom{a+b}{c} T^\fl(a+b+1)\ot T^\fm(a+1,c-a+1),\\
\ncic{5}=&\, T^\fl_{a-c-1}(b+1,c+1)\ot T^\fm(c+2)\\
=&\, (-1)^{a-c-1} \binom{a+b}{c} T^\fl(a+b-c,c+1)\ot T^\fm(c+2).
\end{align*}
Note that $\ncic{5}\,${} is clearly unramified while
$\ncic{2}\,${} is unramified by Lemma~\ref{lem:MtVall>1} since $a\ge1$ and $c-a+1\ge 2$ if $c\ge a+1$.
Thus by putting the ramified contributions together we get
$\ncic{1}+\ncic{3}+\ncic{4}=T^\fl(a+b+1)\ot X$ where
\begin{align*}
X:=&\,
(-1)^a \binom{a+b}{a} \Big( t^\fm(1,c+1)-S^\fm(1,c+1)\Big)\\
&\, -(-1)^a 2 \binom{a+b}{a} \delta_{a=c}\Big(t^\fm(a+1,1)-T^\fm(a+1,1)\Big) \\
&\,+ (-1)^{c} \binom{a+b}{c} \delta_{c+1>a} T^\fm(a+1,c-a+1).
\end{align*}
By Lemma~\ref{lem:Charlton} and Lemma~\ref{lem:D1S1bfl}(2)
\begin{equation} \label{equ:D1t1c+1}
D_1 t^\fm(1,c+1)=2 \log^\fl 2\ot t^\fm(c+1), \qquad D_1 S^\fm(1,c+1)=-2 \log^\fl 2\ot t^\fm(c+1).
\end{equation}
If $c\ne a$ then $D_1 t^\fm(a+1,c-a+1)=D_1 T^\fm(a+1,c-a+1)=0$ by
Lemma~\ref{lem:Charlton} and Lemma~\ref{lem:D1S1bfl}(1). Hence,
\begin{align*}
D_1 X= (-1)^a (2+2) \binom{a+b}{a} \log^\fl 2\ot t^\fm(a+1)\ne 0.
\end{align*}
If $a=c$, then we have
\begin{align*}
X=(-1)^a \binom{a+b}{a} \Big( 3 T^\fm(a+1,1)+t^\fm(1,a+1)-S^\fm(1,a+1)-2t^\fm(a+1,1)\Big).
\end{align*}
Thus if $a=c$ then by \eqref{equ:D1t1c+1} and \eqref{equ:D1T11}
\begin{align*}
D_1 X=9 (-1)^a \binom{a+b}{a} \log^\fl 2\ot t^\fm(a+1)\ne 0.
\end{align*}
In any case we see that $D_{a+b+1}T^\fm(k,l,m)$ is ramified. Hence,
$T^\fm(k,l,m)$ is ramified when $k\equiv l \pmod{2}$.

By the same argument we see that if $(k,l,m)=(E,E,O)$ then
\begin{align*}
D_{a+b+1} S^\fm(k,l,m)=&\, (-1)^a \binom{a+b}{a} T^\fl(a+b+1)\ot \Big( \tz^\fm(1,c+1)-T^\fm(1,c+1)\Big)\\
&\,-(-1)^a \binom{a+b}{a} T^\fl(a+b+1)\ot \delta_{c>a+1} \tz^\fm(a+1,c-a+1) \\
&\,-(-1)^a 2 \binom{a+b}{a} T^\fl(a+b+1)\ot \delta_{a=c} \Big(\tz^\fm(a+1,1)-S^\fm(a+1,1)\Big)\\
&\,+(-1)^{c} \binom{a+b}{c} T^\fl(a+b+1)\ot \delta_{c+1>a} S^\fm(a+1,c-a+1)\\
&\,+(-1)^{a-c-1} \binom{a+b}{c} T^\fl(a+b-c,c+1)\ot \delta_{c+1<a} S^\fm(c+2),
\end{align*}
where the second and the fifth terms are again unramified. Now that $a$ is odd and $c$ is even
the third term cannot appear. By Thm.~\ref{thm-distribution} the only ramified terms are produced by
$(-1)^a \binom{a+b}{a} T^\fl(a+b+1)\ot Y$, where
\begin{equation*}
 Y=\binom{a+b}{a} T^\fm(1,c+1)+ \binom{a+b}{c} \delta_{c+1>a} S^\fm(a+1,c-a+1)
\end{equation*}
for which by \eqref{equ:Dw-1DblTwtEven} and \eqref{equ:Dw-1DblSwtEven} we have
\begin{align*}
D_{c+1}Y
=\bigg(2\binom{a+b}{a}- \delta_{c>a} 2\binom{a+b}{c}\binom{c}{a} \bigg) T^\fl(c+1) \ot \log^\fm 2 \\
=2\binom{a+b}{a}\bigg(1- \delta_{c>a} \binom{b}{c-a} \bigg) T^\fl(c+1) \ot \log^\fm 2 \ne 0
\end{align*}
since $b\ne c-a$ as they have different parities. Thus $S^\fm(k,l,m)$ is ramified in this case, too.

\medskip
\noindent
\textbf{Case II}. (This includes $f^\fm(E,O,O_{>1})$ and $f^\fm(O_{>1},E,E)$.)
Suppose $a\not\equiv b\equiv c \pmod{2}$. Then we have the following possible cuts for $D_{b+c+1}T^\fm(k,l,m)$.

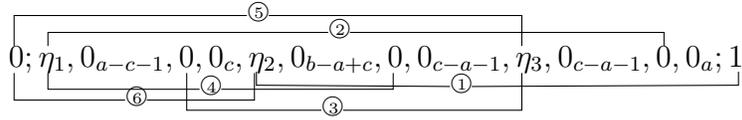
\begin{center}
\begin{tikzpicture}[scale=0.9]
\node (A0) at (0.05,0) {$0;\eta_1,0_{a-c-1},0,0_c,\eta_2,0_{b-a+c},0,0_{c-a-1},\eta_3,0_{c-a-1},0,0_a;1$};
\node (A1) at (1.3,-0.4) {${}$};
\node (A2) at (-0.5,0.4) {${}$};
\node (A3) at (-0.6,-0.75) {${}$};
\node (A4) at (-2.4,-0.44) {${}$};
\node (A5) at (-1.7,0.65) {${}$};
\node (A6) at (-3.5,-0.6) {${}$};
\draw (-1.72,-.18) to (-1.72,-0.38) to (A1) node {$\cic{1}$} to (5.4,-0.38) to (5.4,-0.18);
\draw (-4.8,0.2) to (-4.8,0.4) to (A2) node {$\cic{2}$} to (4.3,0.4) to (4.3,0.2);
\draw (-2.75,-0.2) to (-2.75,-0.75) to (A3) node {$\cic{3}$} to (2.2,-0.75) to (2.2,-0.2);
\draw (-4.8,-0.2) to (-4.8,-0.44) to (A4) node {$\cic{4}$} to (0.3,-0.44) to (0.3,-0.2);
\draw (-5.3,0.2) to (-5.3,0.65) to (A5) node {$\cic{5}$} to (2.2,0.65) to (2.2,0.2);
\draw (-5.3,-0.2) to (-5.3,-0.6) to (A6) node {$\cic{6}$} to (-1.75,-0.6) to (-1.75,-0.2);
\node (A) at (0,-1.6) {Possible cuts of $D_{b+c+1} f^\fm(k,l,m)$ ($f=T$ or $S$), $a\not\equiv b\equiv c\pmod{2}$.};
\end{tikzpicture}
\end{center}

Thus
\begin{align*}
&\, D_{b+c+1}T^\fm(k,l,m)\\
=&\, \sum_{\eta_1,\eta_2,\eta_3=\pm1}\eta_1\eta_2\eta_3 \bigg[I^\fl(\eta_2;0_b,\eta_3,0_c;1) \ot I^\fm(0;\eta_1,0_a,\eta_2;1) \\
&\,+ \gd_{c\ge a+1} \, I^\fl(\eta_1;0_a,\eta_2,0_b,\eta_3,0_{c-a-1};0) \ot I^\fm(0;\eta_1,0_{a+1};1) \\
&\,+ \gd_{c<a} \, I^\fl(0;0_c,\eta_2,0_b;\eta_3) \ot I^\fm(0;\eta_1,0_{a-c},\eta_3,0_c;1) \\
&\,+ \gd_{c\le a\le b+c} \, I^\fl(\eta_1;0_a,\eta_2,0_{b-a+c};0) \ot I^\fm(0;\eta_1,0_{a-c},\eta_3,0_c;1) \\
&\,+ \gd_{a+1=c}\, I^\fl(0;\eta_1,0_a,\eta_2,0_b;\eta_3) \ot I^\fm(0;\eta_3,0_c;1) \\
&\,+ \gd_{a=b+c,2|a} \, I^\fl(0;\eta_1,0_a;\eta_2) \ot I^\fm(0;\eta_2,0_b,\eta_3,0_c;1)\bigg] \\
=&\, \sum_{\eta_1,\eta_2,\eta_3=\pm1} \bigg[
\Big( \eta_1\eta_2\eta_3 I^\fl(0;0_b,\eta_3,0_c;1)+\eta_1\eta_3I^\fl(1;0_b,\eta_3,0_c;0) \Big) \ot I^\fm(0;\eta_1,0_a,\eta_2;1) \\
&\,- \gd_{c\ge a+1} \, \eta_1\eta_2\eta_3 I^\fl(0;0_{c-a-1},\eta_3,0_b,\eta_2,0_a;1) \ot I^\fm(0;\eta_1,0_{a+1};1) \\
&\,+ \gd_{c<a} \, \eta_1 \eta_2 I^\fl(0;0_c,\eta_2,0_b;1) \ot I^\fm(0;\eta_1,0_{a-c},\eta_3,0_c;1) \\
&\,+ \gd_{c< a\le b+c} \, \eta_2\eta_3 I^\fl(1;0_a,\eta_2,0_{b-a+c};0) \ot I^\fm(0;\eta_1,0_{a-c},\eta_3,0_c;1) \\
&\,+ \gd_{a+1=c} \, \eta_1\eta_2\eta_3 I^\fl(0;\eta_1\eta_3,0_a,\eta_2\eta_3,0_b;\eta_3) \ot I^\fm(0;\eta_3,0_c;1) \\
&\,+ \gd_{a=b+c,2|a} \, \eta_1 \eta_3 I^\fl(0;\eta_1,0_a;1) \ot I^\fm(0;\eta_2,0_b,\eta_3,0_c;1) \bigg]\\
=&\, T^\fl_b(c+1)\ot T^\fm(a+1,1)- T^\fl_c(b+1)\ot t^\fm(a+1,1) \\
&\,-\gd_{c\ge a+1}\, T^\fl_{c-a-1}(b+1,a+1)\ot T^\fm(a+2) \\
&\,+ \gd_{c<a}\, T^\fl_c(b+1) \ot t^\fm(a-c+1,c+1)\\
&\,- \gd_{c<a\le b+c}\, T^\fl_{b-a+c}(a+1) \ot S^\fm(a-c+1,c+1)\\
&\,+\gd_{a+1=c}\, T^\fl(a+1,b+1)\ot T^\fm(a+2) \\
&\,+\gd_{a=b+c,2|a} \, T^\fl(a+1)\ot S^\fm(b+1,c+1) \\
\equiv &\, (-1)^b T^\fl(b+c+1) \ot \bigg[ \binom{b+c}{b} \Big(T^\fm(a+1,1)- t^\fm(a+1,1) \Big) \\
&\, \qquad \qquad \qquad + \bigg((-1)^b \gd_{a=b+c,2|a}-\gd_{c<a\le b+c} \, \binom{b+c}{a} \bigg) S^\fm(a-c+1,c+1) \bigg]
\end{align*}
modulo the unramified terms $\cdots \ot T^\fm(a+2)$ by Lemma~\ref{lem:singleTs} and Lemma~\ref{prop:DTwtOdd},
unramified term $T^\fl_c(b+1) \ot t^\fm(a-c+1,c+1)$ by Lemma~\ref{lem:MtVall>1}.
Note that if $c<a$ then $a-c+1,c+1\ge 2$ and therefore $D_1 S^\fm(a-c+1,c+1) =0$ by Lemma~\ref{lem:D1}. But
\begin{equation*}
D_1 \Big(T^\fm(a+1,1)- t^\fm(a+1,1) \Big)= 2 \log^\fl 2\ot T^\fm(a+1)\ne0.
\end{equation*}
This implies that $T^\fm(k,l,m)$ is ramified in this case.

Similarly, for $S^\fm(k,l,m)$ by assuming $b$ and $c$ are both odd we can compute to get
\begin{align*}
&\, D_{b+c+1}S^\fm(k,l,m)\\
=&\, T^\fl_b(c+1)\ot S^\fm(a+1,1)- T^\fl_c(b+1)\ot \tz^\fm(a+1,1) \\
&\,-\gd_{c\ge a+1} \, T^\fl_{c-a-1}(b+1,a+1)\ot \tz^\fm(a+2) \\
&\,+ \gd_{c<a} \, T^\fl_c(b+1) \ot \tz^\fm(a-c+1,c+1)\\
&\,- \gd_{c< a\le b+c}\, T^\fl_{b-a+c}(a+1) \ot T^\fm(a-c+1,c+1)\\
&\,+\gd_{a=c+1} \, S^\fl(a+1,b+1)\ot \tz^\fm(c+1) \\
&\,+\gd_{a=b+c,2|a} \, T^\fl(a+1)\ot T^\fm(b+1,c+1) \\
\equiv &\, (-1)^b T^\fl(b+c+1)\ot \bigg[ \binom{b+c}{b} \Big(S^\fm(a+1,1) -\tz^\fm(a+1,1)\Big) \\
&\, \qquad \qquad- \gd_{c<a\le b+c}\, \binom{b+c}{a} T^\fm(a-c+1,c+1) +(-1)^b \gd_{a=b+c,2|a}\, T^\fm(b+1,c+1) \bigg]
\end{align*}
modulo all the unramified terms. We can see that $S^\fm(k,l,m)$ is ramified since
\begin{equation*}
D_1\Big(S^\fm(a+1,1) -\tz^\fm(a+1,1)\Big) = 2\log^\fl 2\ot S^\fm(a+1)\ne 0
\end{equation*}
using \eqref{equ:D1Sk1} and Lemma~\ref{lem-D1tzk1}.
Note that in this case, $a-c+1,c+1,b+1\ge 2$.

\medskip
\noindent
\textbf{Case III}.  (This includes $f^\fm(E,O_{>1},E)$.)
Suppose $b\not\equiv a\equiv c \equiv 1\pmod{2}$ and $b\ge 2$.

(i) First we consider $D_{a+2}$ when $a\le c$.

\begin{center}
\begin{tikzpicture}[scale=0.9]
\node (A0) at (0.05,0) {$0;\eta_1,0_{b-2},0,0_{a-b+1},\eta_2,0,0,0_{b-2},\eta_3,0_{a-b+1},0,0_{c-a+b-2};1$};
\node (A1) at (-2.6,-0.4) {${}$};
\node (A2) at (-.9,0.4) {${}$};
\node (A3) at (1,-0.6) {${}$};
\node (A4) at (2.7,0.6) {${}$};
\draw (-4.8,-0.25) to (-4.8,-0.4) to (A1) node {$\cic{1}$} to (-0.3,-0.4) to (-0.3,-0.25);
\draw (-3.3,0.25) to (-3.3,0.4) to (A2) node {$\cic{2}$} to (1.2,0.4) to (1.2,0.25);
\draw (-1.3,-0.25) to (-1.3,-0.6) to (A3) node {$\cic{3}$} to (3.3,-0.6) to (3.3,-0.25);
\draw (0.2,0.25) to (0.2,0.6) to (A4) node {$\cic{4}$} to (5.45,0.6) to (5.45,0.25);
\node (A) at (0,-1.2) {Possible cuts of $D_{a+2} f^\fm(k,l,m)$ ($f=T$ or $S$), $a\le c$, $b\not\equiv a\equiv c \equiv 1\pmod{2}$.};
\end{tikzpicture}
\end{center}

Thus
\begin{align*}
&\, D_{a+2}T^\fm(k,l,m)\\
=&\, \sum_{\eta_1,\eta_2,\eta_3=\pm1}\eta_1\eta_2\eta_3 I^\fl(\eta_1;0_a,\eta_2,0;0) \ot I^\fm(0;\eta_1,0_{b-1},\eta_3,0_c;1) \\
&\,+ \gd_{b\le a+1} \sum_{\eta_1,\eta_2,\eta_3=\pm1}\eta_1\eta_2\eta_3 I^\fl(0;0_{a-b+1},\eta_2,0_b;\eta_3) \ot I^\fm(0;\eta_1,0_{b-1},\eta_3,0_c;1) \\
&\,+ \gd_{b\le a+1\le b+c} \sum_{\eta_1,\eta_2,\eta_3=\pm1}\eta_1\eta_2\eta_3 I^\fl(\eta_2;0_b,\eta_3,0_{a-b+1};0) \ot I^\fm(0;\eta_1,0_a,\eta_2,0_{c-a+b-1};1) \\
&\,+ \gd_{c\le a+1\le b+c} \sum_{\eta_1,\eta_2,\eta_3=\pm1}\eta_1\eta_2\eta_3 I^\fl(0;0_{a-c+1},\eta_3,0_c;1) \ot I^\fm(0;\eta_1,0_a,\eta_2,0_{b-a+c-1};1) \\
=&\, -T^\fl_1(a+1) \ot S^\fm(b,c+1)
 + \gd_{b\le a+1} \, T^\fl_{a-b+1}(b+1) \ot t^\fm(b,c+1) \\
&\,- \gd_{b\le a+1\le b+c} \, T^\fl_{a-b+1}(b+1) \ot t^\fm(a+1,c-a+b)
 + \gd_{c=a}\, T^\fl_{1}(c+1) \ot T^\fm(a+1,b) \\
=&\, T^\fl(a+2) \ot \bigg[(a+1) S^\fm(b,c+1)
 + \gd_{b\le a+1}\, \binom{a+1}{b} t^\fm(b,c+1) \\
&\, \qquad - \gd_{b\le a+1\le b+c}\, \binom{a+1}{b} t^\fm(a+1,c-a+b)
-\gd_{c=a}\, (a+1) T^\fm(c+1,b)\bigg]
\end{align*}
by Lemma~\ref{lem:singleTs} since $a\le c$. But
\begin{align*}
&\, D_{b+c} S^\fm(b,c+1) =\sum_{\eta_2=\pm1} \eta_2 I^\fl(\bar1;0_{b-1},\eta_2,0_c;1)\ot I^\fm(0;\bar1;1)\\
=&\,\Big(T^\fl_c(b)+T^\fl_{b-1}(c+1)\Big)\ot I^\fm(0;\bar1;1)=-2\binom{b+c-1}{c} T^\fl(b+c)\ot \log^\fm 2\ne 0
\end{align*}
while $D_{b+c} t^\fm(b,c+1)=D_{b+c} t^\fm(a+1,c-a+b)=0$ by Lemma~\ref{lem:Dw-1t} since $c+1\ge2$ and $c-a+b\ge 2$. If $a=c$ we need to compute one extra contribution:
\begin{align*}
&\, D_{b+c} T^\fm(c+1,b) = -\sum_{\eta_2=\pm1} \eta_2 I^\fl(\bar1;0_c,\eta_2,0_{b-1};1)\ot I^\fm(0;\bar1;1)\\
=&\,-\Big(T^\fl_c(b)+T^\fl_{b-1}(c+1)\Big)\ot I^\fm(0;\bar1;1)= 2\binom{b+c-1}{c} T^\fl(b+c)\ot \log^\fm 2
\end{align*}
so that
\begin{align*}
D_{b+c} \Big(S^\fm(b,c+1)-T^\fm(c+1,b)\Big)=-4\binom{b+c-1}{c} T^\fl(b+c)\ot \log^\fm 2\ne 0.
\end{align*}
Thus $D_{a+2}T^\fm(k,l,m)$ is ramified in this case.

Similarly,
\begin{align*}
&\, D_{a+2}S^\fm(k,l,m)\\
=&\, -T^\fl_1(a+1) \ot T^\fm(b,c+1)
 + \gd_{b\le a+1}\, T^\fl_{a-b+1}(b+1) \ot \tz^\fm(b,c+1) \\
&\,- \gd_{b\le a+1\le b+c}\, T^\fl_{a-b+1}(b+1) \ot \tz^\fm(a+1,c-a+b)
 + \gd_{c=a} \, T^\fl_{1}(c+1) \ot S^\fm(a+1,b) \\
\equiv &\, T^\fl(a+2) \ot \bigg[(a+1)\Big( T^\fm(b,c+1) - \gd_{c=a}\, S^\fm(a+1,b)\Big)-\gd_{a+1=b+c}\, \binom{a+1}{b}\tz^\fm(a+1,1)\bigg]
\end{align*}
modulo unramified terms. This is clearly ramified if $c\ne a$ and $a+1\ne b+c$ by Prop.~\ref{prop:DTwtEven}.
If $c=a$ then $a+1\ne b+c$ since $b\ge 2$, and therefore by \eqref{equ:Dw-1DblTwtEven} and \eqref{equ:Dw-1DblSwtEven}
\begin{align*}
 D_{a+b}\Big( T^\fm(b,a+1)-S^\fm(a+1,b)\Big)= 4\binom{a+b-1}{a}T^\fl(a+b) \ot \log^\fm 2
\end{align*}
which is ramified. If $a+1= b+c$ then $a\ne c$ and therefore
\begin{align*}
 D_1\Big((a+1) T^\fm(b,a+1)-\binom{a+1}{b}\tz^\fm(a+1,1)\Big)=\binom{a+1}{b}S^\fm(a+1) \ot \log^\fm 2 \ne 0.
\end{align*}
Consequently, $D_{a+2}S^\fm(k,l,m)$ is always ramified in this case.

\medskip
(ii) Next we consider $D_{c+2}$ when $c\le a-2$.

\begin{center}
\begin{tikzpicture}[scale=0.9]
\node (A0) at (0.05,0) {$0;\eta_1,0_{a+b-c-2},0,0_{c-b+1},\eta_2,0_{c-a+1},0,0_{a+b-c-2},\eta_3,0_{c-b+1},0,0_{b-2};1$};
\node (A1) at (-2.6,-0.4) {${}$};
\node (A2) at (-.3,0.4) {${}$};
\node (A3) at (1.7,-0.6) {${}$};
\node (A4) at (4.0,0.6) {${}$};
%\draw (-5.75,-0.25) to (-5.75,-0.4) to (A1) node {$\cic{4}$} to (0.6,-0.4) to (0.6,-0.25);
\draw (-3.35,0.25) to (-3.35,0.4) to (A2) node {$\cic{3}$} to (2.85,0.4) to (2.85,0.25);
\draw (-1.4,-0.25) to (-1.4,-0.6) to (A3) node {$\cic{2}$} to (4.9,-0.6) to (4.9,-0.25);
\draw (1.8,0.25) to (1.8,0.6) to (A4) node {$\cic{1}$} to (6.35,0.6) to (6.35,0.25);
\node (A) at (0,-1.2) {Possible cuts of $D_{c+2}T^\fm(k,l,m)$, $c\le a-2$, $b\not\equiv a\equiv c \equiv 1\pmod{2}$.};
\end{tikzpicture}
\end{center}

We have
\begin{align*}
&\, D_{c+2}T^\fm(k,l,m)\\
=&\, \sum_{\eta_1,\eta_2,\eta_3=\pm1}\eta_1\eta_2\eta_3 I^\fl(0;0,\eta_3,0_c;1) \ot I^\fm(0;\eta_1,0_a,\eta_2,0_{b-1};1) \\
&\,+ \gd_{c\ge b-1} \sum_{\eta_1,\eta_2,\eta_3=\pm1}\eta_1\eta_2\eta_3 I^\fl(\eta_2;0_b,\eta_3,0_{c-b+1};0) \ot I^\fm(0;\eta_1,0_a,\eta_2,0_{b-1};1) \\
&\,+ \gd_{a\ge c-b+1\ge 0} \sum_{\eta_1,\eta_2,\eta_3=\pm1}\eta_1\eta_2\eta_3 I^\fl(0;0_{c-b+1},\eta_2,0_b;\eta_3) \ot I^\fm(0;\eta_1,0_{a+b-c-1},\eta_3,0_c;1) \\
=&\, T^\fl_1(c+1) \ot T^\fm(a+1,b) - \gd_{c\ge b-1} \, T^\fl_{c-b+1}(b+1) \ot t^\fm(a+1,b)\\
&\,+ \gd_{a\ge c-b+1\ge 0}\, T^\fl_{c-b+1}(b+1)\ot t^\fm(a+b-c,c+1) \\
=&\, T^\fl(c+2) \ot \bigg[-(c+1) T^\fm(a+1,b)
- \gd_{c\ge b-1}\, \binom{c+1}{b} t^\fm(a+1,b) \\
&\, \qquad \qquad + \gd_{a\ge c-b+1\ge 0} \, \binom{c+1}{b} t^\fm(a+b-c,c+1) \bigg]
\end{align*}
by Lemma~\ref{lem:singleTs} since $c\le a-2$. But
\begin{align*}
D_{a+b} T^\fm(a+1,b) = &\,-\sum_{\eta_2=\pm1} \eta_2 I^\fl(\bar1;0_a,\eta_2,0_{b-1};1)\ot I^\fm(0;\bar1;1)\\
=&\,-\Big(T^\fl_a(b)+T^\fl_{b-1}(a+1)\Big)\ot I^\fm(0;\bar1;1)\\
=&\, 2\binom{a+b-1}{a} T^\fl(a+b)\ot \log^\fm 2\ne 0
\end{align*}
while $D_{a+b} t^\fm(a+1,b)=D_{a+b} t^\fm(a+b-c,c+1)=0$. Thus $D_{c+2}T^\fm(k,l,m)$ is ramified in this case.

Similarly,
\begin{align*}
 D_{c+2}S^\fm(k,l,m)
=&\, T^\fl_1(c+1) \ot S^\fm(a+1,b) - \gd_{c\ge b-1}\, T^\fl_{c-b+1}(b+1) \ot \tz^\fm(a+1,b)\\
&\,+ \gd_{a\ge c-b+1\ge 0}\, T^\fl_{c-b+1}(b+1)\ot \tz^\fm(a+b-c,c+1)
\end{align*}
whose first term is the only ramified term so the expression as a whole is ramified.

Combining (i) and (ii) we see that both $T^\fm(k,l,m)$ and $S^\fm(k,l,m)$ are ramified by Thm.~\ref{thm-Glanois}.

\bigskip
We have now exhausted all the possibilities for the triple $T$-values and triple $S$-values, which complete the proof
of Thm.~\ref{thm:MTV} and Thm.~\ref{thm:MSV}.

\section{Ramified and unramified MtVs}\label{sec:MtVs}
In this section, we shall prove Thm.~\ref{thm:unramfiedMtV} and Thm.~\ref{thm:ramfiedMtV} using Glanois's descent theory. We start with the following result due to Murakami.

\begin{lem}\label{lem:MtVall>1} \emph{(cf. \cite[Thm.~8]{Murakami2021})}
If positive integers $k_1\ge 2,\dots,k_d\ge 2$ then the motivic multiple $t$-value $t^\fm(k_1,\dots,k_d)$ is unramified.
\end{lem}

In \cite{Charlton2021}, Charlton improved some of Murakami's results and proved the following.
For $\bfk = (k_1,\ldots,k_d) \in \N^d$, we put $\bfk_{i,j} = (k_i, \ldots, k_j)$ and
set $\bfk_{j,i}$ to be the reversal of $\bfk_{i,j}$ for all $1\le i\le j\le d$.

\begin{lem}\label{lem:Charlton} \emph{(cf. \cite[Prop.~5.8]{Charlton2021})}
Let $\bfk=(k_1,\dots,k_d)$ be a $d$-tuple of positive integers. Then
\begin{equation}\label{equ:D1MtVs}
 D_1 t^\fm(\bfk)=\log^\fl 2\ot \Big(2\delta_{k_1=1} \, t^\fm(\bfk_{2,d})-\delta_{k_d=1} \,  t^\fm(\bfk_{1,d-1})\Big),
\end{equation}
where $\gd_\bullet=1$ if $\bullet$ is true and $0$ otherwise.
\end{lem}

\begin{prop}\label{prop:MtV1s}
For integers $a,b$ with $a+b\ge 1$ the motivic $t$-values $t^\fm(1_a,\bfk,1_b)$ are ramified for all $\bfk$.
\end{prop}
\begin{proof}
Without loss of generality we may assume $\bfk$ neither starts or ends with 1.
If $\bfk$ is empty and $b=0$ then from Lemma~\ref{lem:singleTs} we may assume $a\ge 2$.
By \cite[Prop.~5.8]{Charlton2021} we get
\begin{align*}
D_1 t^\fm(1_a)=&\, \log^\fl 2\ot t^\fm(1_{a-1})\ne 0.
\end{align*}
Thus the proposition follows from Thm.~\ref{thm-Glanois}.

Now suppose $\bfk$ is nonempty. By \eqref{equ:D1MtVs} we get
\begin{equation*}
D_1t^\fm(1_a,\bfk,1_b)=\log^\fl 2\ot \Big(2 \delta_{a>0} t^\fm(1_{a-1},\bfk,1_b)-\delta_{b>0} t^\fm(1_a,\bfk,1_{b-1})\Big).
\end{equation*}
Defining $D_1^{(0)}=D_1$ and $D_1^{(\ell)}=\text{id}\ot D_1^{(\ell-1)}$ for all $\ell\ge1$ we get
\begin{equation*}
D_1^{(\ell)} t^\fm(1_a,\bfk,1_b)=(\log^\fl 2)^{\ot (\ell-1)}\ot \sum_{\substack{i+j=\ell\\ i,j\ge 0}}
\gd_{a\ge i,b\ge j}\, C_{i,j}^{(\ell)} t^\fm(1_{a-i},\bfk,1_{b-j}),
\end{equation*}
where the coefficients $C_{i,j}^{(\ell)}$ satisfy the recurrence relation
\begin{equation*}
C_{i,j}^{(\ell)}=2C_{i-1,j}^{(\ell-1)}-C_{i,j-1}^{(\ell-1)}.
\end{equation*}
With substitution $C_{i,j}^{(\ell)}=2^i(-1)^j B_{i,j}^{(\ell)}$ the recurrence becomes
\begin{equation*}
B_{i,j}^{(\ell)}=B_{i-1,j}^{(\ell-1)}+B_{i,j-1}^{(\ell-1)}
\end{equation*}
so that $B_{i,j}^{(\ell)}=\binom{\ell}{i}$ are just the binomial coefficients. Hence,
$C_{i,j}^{(\ell)}=2^i(-1)^j \binom{\ell}{i}$ and therefore
\begin{align*}
D_1^{(a+b)} t^\fm(1_a,\bfk,1_b)=(\log^\fl 2)^{\ot (a+b-1)}\ot 2^a(-1)^b \binom{a+b}{a} t^\fm(\bfk) \ne 0.
\end{align*}
Repeatedly using Thm.~\ref{thm-Glanois} we see that none of the $D_1^{(\ell)}$ ($\ell<a+b$) vanishes so that
$t^\fm(1_a,\bfk,1_b)$ is ramified.
\end{proof}

The following fact will be useful in our computations later.
\begin{lem}\label{lem:Dw-1t}
For any admissible $\bfk$ of weight $w$ we have $D_{w-1} t^\fm(\bfk)=0$.
\end{lem}
\begin{proof} Suppose $\bfk=(n_1+1,\dots,n_d+1)$, $n_d\ge 1$. Then
\begin{align*}
D_{w-1} t^\fm(\bfk)=&\,\sum_{\eta_1,\dots,\eta_d=\pm1} \eta_1 D_{w-1}I^\fm(0;\eta_1,0_{n_1},\dots,,\eta_d,0_{n_d};1)\\
=&\,\sum_{\eta_2,\dots,\eta_d=\pm1} I^\fl(\bar1;0_{n_1},\eta_2,0_{n_2},\dots,\eta_d,0_{n_d};1)\ot \log^\fm 2=0
\end{align*}
by Lemma~\ref{Murakami-lemma9}.
\end{proof}

In fact, Charlton proved a detailed computation of the derivation $D_r$ on motivic MtVs as follows.
Lemma~\ref{lem:Charlton} above is only one of its corollaries.

\begin{thm}\label{thm:Charlton} \emph{(cf. \cite[Prop.~5.6]{Charlton2021})}
Let $\bfk = (k_1,\ldots,k_d) \in \N^d$. Then
\begin{align}
&\label{eqn:dr:deconcat} D_r \big( t^\fm(\bfk)\big) =
\sum_{1 \le j \le d} \delta_{|\bfk_{1,j}|=r} t^\fl(\bfk_{1,j}) \ot t^\fm(\bfk_{j+1,d} ) \\[1ex]
& \label{eqn:dr:0eps} + \!\! \sum_{1 \le i < j \le d}
	\delta_{|\bfk_{i+1,j}| \le r < |\bfk_{i,j}|-1} \binom{\zeta^\fl_{r-|\bfk_{i+1,j}|}(\bfk_{i+1,j})}{-\delta_{r=1} \log^\fl 2} \ot t^\fm(\bfk_{1,i-1},|\bfk_{i,j}| - r, \bfk_{j+1,d}) \\[1ex]
& \label{eqn:dr:eps0} - \!\! \sum_{1 \le i < j \le d}
			\delta_{|\bfk_{i,j-1}| \le r <|\bfk_{i,j}|-1} \binom{\zeta^\fl_{r-|\bfk_{i,j-1}|}(\bfk_{j-1,i})}{-\delta_{r=1} \log^\fl 2}\ot t^\fm(\bfk_{1,i-1},|\bfk_{i,j}| - r, \bfk_{j+1,d}).
\end{align}		
\end{thm}

The next result will be crucial in our proof of some families of unramified MtVs.
\begin{prop}\label{prop:unrmMtVmodProd}
Suppose $\bfk=(k_1,\ldots,k_d)$ where $k_1,\ldots,k_d\ge 2$. Then
\begin{equation*}
t^\fm(\bfk,1)+t^\fm(\bfk)\log^\fm 2, \quad
\tz^\fm(\bfk,1)+\tz^\fm(\bfk)\log^\fm 2,\quad
\tz^\fm(1,\bfk,1)+\tz^\fm(1,\bfk)\log^\fm 2
\end{equation*}
are all unramified.
\end{prop}
\begin{proof} Let's consider the motivic MtVs first.
By Glanois's result Thm.~\ref{thm-Glanois} we only need to show that $D_r\Big(t^\fm(\bfk,1)+ t^\fm(\bfk)\log^\fm 2\Big)$ is unramified for odd $1<r\le |\bfk|$ and vanishes for $r=1$. As $D_r$'s are derivations,
by Lemma~\ref{lem:Charlton} we clearly have
\begin{equation*}
D_1\Big(t^\fm(\bfk,1)+ t^\fm(\bfk)\log^\fm 2\Big)=-\log^\fl 2 \ot t^\fm(\bfk)+D_1\Big(t^\fm(\bfk)\log^\fm 2 \Big)
=(1\ot \log^\fm) D_1 t^\fm(\bfk)=0
\end{equation*}
by Lemma~\ref{lem:MtVall>1}.

Now assume $r$ is odd and $1<r\le |\bfk|$. Let $k_{d+1}=1$. Then by Thm.~\ref{thm:Charlton} we see that
(after setting $\bfk_{d+1,d}=\emptyset$ and $t^\fm(\emptyset)=1$)
\begin{align*}
 & D_r\Big(t^\fm(\bfk,1)+ t^\fm(\bfk)\log^\fm 2\Big) =
 \sum_{1 \le j \le d} \delta_{|\bfk_{1,j}|=r} t^\fl(\bfk_{1,j}) \ot
 \Big(t^\fm(\bfk_{j+1,d},1)+t^\fm(\bfk_{j+1,d})\log^\fm 2 \Big) \\[1ex]
& + \!\! \sum_{1 \le i < j \le d}
	\delta_{|\bfk_{i+1,j}| \le r < |\bfk_{i,j}|-1} \zeta^\fl_{r-|\bfk_{i+1,j}|} (\bfk_{i+1,j}) \ot \\[1ex]
& \hskip2cm \Big( t^\fm(\bfk_{1,i-1},|\bfk_{i,j}| - r, \bfk_{j+1,d},1)
  +t^\fm(\bfk_{1,i-1},|\bfk_{i,j}| - r, \bfk_{j+1,d})\log^\fm 2 \Big) \\[1ex]
& - \!\! \sum_{1 \le i < j \le d}
			\delta_{|\bfk_{i,j-1}| \le r <|\bfk_{i,j}|-1} \zeta^\fl_{r-|\bfk_{i,j-1}|} (\bfk_{j-1,i}) \ot \\[1ex]
& \hskip2cm \Big( t^\fm(\bfk_{1,i-1},|\bfk_{i,j}| - r, \bfk_{j+1,d},1)
   + t^\fm(\bfk_{1,i-1},|\bfk_{i,j}| - r, \bfk_{j+1,d}) \log^\fm 2 \Big) \\
& + \!\! \sum_{1 \le i\le d}
	\delta_{|\bfk_{i+1,d+1}| \le r < |\bfk_{i,d}|} \zeta^\fl_{r-|\bfk_{i+1,d+1}|}(\bfk_{i+1,d+1}) \ot t^\fm(\bfk_{1,i-1},|\bfk_{i,d}|+1 - r) .
\end{align*}		
Every term in the first three summands is unramified by induction while the last sum is also unramified by Lemma~\ref{lem:MtVall>1}.

The proof that $\tz^\fm(\bfk,1)+\tz^\fm(\bfk)\log^\fm 2 $ and
$\tz^\fm(1,\bfk,1)+\tz^\fm(1,\bfk)\log^\fm 2$ are unramified is completely similar using
next lemma so we leave it to the interested reader.
\end{proof}

\begin{lem}\label{lem-D1tzk1}
For any nontrivial $\bfk$ we have
\begin{equation*}
D_1\tz^\fm_a(\bfk,1)=-\log^\fl 2\ot \tz^\fm_a(\bfk).
\end{equation*}
\end{lem}
\begin{proof}
Observe that nontrivial cuts of $D_1\tz^\fm_a(\bfk)$ can appear when $\bfk$ has unit components. First, assume
such a component appears in the following block:
\begin{equation*}
\sum_{\eta_1,\dots,\eta_d=\pm1} D_1 I^\fm(0;\bfga,0,\eta_i,\eta_{i+1},\dots,\eta_{j-1},\eta_j,0,\bfgb;1)
\end{equation*}
for some strings $\bfga$ and $\bfgb$. We observe that nontrivial cuts produced from the above block $0,\eta_i,\dots,\eta_j,0$
lead to
\begin{align*}
& \sum_{\eta_1,\dots,\eta_d=\pm1} \Big(I^\fl(0;\eta_i;\eta_{i+1})\ot I^\fm(0;\bfga,0,\eta_{i+1},\eta_{i+2},\dots,\eta_{j-1},\eta_j,0,\bfgb;1)\\
& +I^\fl(\eta_{j-1};\eta_j;0)\ot I^\fm(0;\bfga,0,\eta_i,\eta_{i+1},\dots,\eta_{j-2},\eta_{j-1},0,\bfgb;1) \Big)\\
& +\sum_{\eta_1,\dots,\eta_d=\pm1} \sum_{\ell=i+1}^{j-1} I^\fl(\eta_{\ell-1};\eta_\ell;\eta_{\ell+1})\ot I^\fm(0;\bfga,0,\eta_i,\dots,\widehat{\eta_\ell},\dots,\eta_j,0,\bfgb;1).
\end{align*}
The first two terms cancel if we apply to the second term the substitution
$\eta_i\to \eta_{i+1}\to \cdots\to \eta_{j-1} \to \eta_j \to \eta_j\eta_i/\eta_{i+1}$.
Further, for each fixed $\ell$, by Lemma~\ref{Murakami-lemma9} the sum
\begin{equation*}
\sum_{\eta_\ell =\pm1} I^\fl(\eta_{\ell-1};\eta_\ell;\eta_{\ell+1})=0.
\end{equation*}
Therefore we see that $D_1\tz^\fm_a(\bfk)=D_1\zeta^\fm_a(\bfk)=0$.

Now if there is a unit component at the end then the above argument shows that all but one term is left after cancelation
and the lemma follows easily.
\end{proof}

We are now ready to prove the following theorem.

\begin{thm}\label{thm:unramfiedMtV1}
Let $n,e\in\N$ and all components of $\bfq$ are at least $2$. Then the motivic MtV $t^\fm(\{2n\}_e,1,\bfq)$ is unramified.
Hence, $t(\{2n\}_e,1,\bfq)$ is unramified.
\end{thm}

\begin{proof}
Let $\bfk=(\{2n\}_e,1,\bfq)$. By Lemma~\ref{lem:Charlton} we see that $D_1 t^\fm(\bfk)=0$. Assume now $1<r<|\bfk|$ is odd. Then by Lemma~\ref{Murakami-lemma9} we only need to consider the cuts of $D_r$ either starting at or ending at 0:

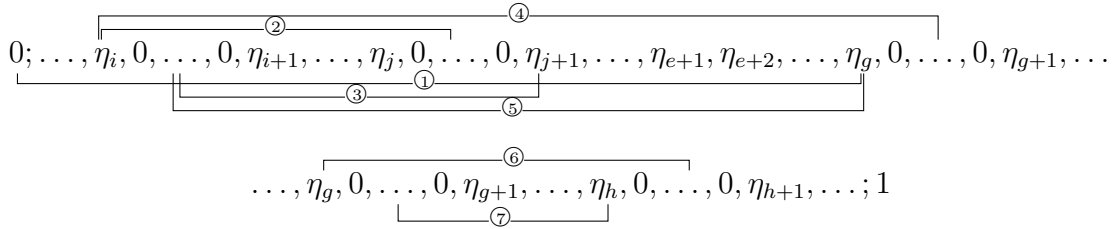
\begin{center}
\begin{tikzpicture}[scale=0.9]
\node (A0) at (0.05,0) {$\ \hskip-5.2cm 0;\dots,\eta_i,0,\dots,0,\eta_{i+1},\dots,\eta_j,0,\dots,0,\eta_{j+1},
 \dots,\eta_{e+1},\eta_{e+2},\dots,\eta_g,0,\dots,0,\eta_{g+1},\dots$};
\node (A1) at (-4.8,-0.4) {${}$};
\node (A2) at (-7,0.4) {${}$};
\node (A3) at (-5.8,-0.6) {${}$};
\node (A4) at (-3.4,0.6) {${}$};
\node (A5) at (-3.5,-0.8) {${}$};
%\node (A8) at (0.9,0.8) {${}$};
\draw (-10.8,-0.25) to (-10.8,-0.4) to (A1) node {$\cic{1}$} to (1.66,-0.4) to (1.66,-0.25);
\draw (-9.56,0.25) to (-9.56,0.4) to (A2) node {$\cic{2}$} to (-4.4,0.4) to (-4.4,0.25);
\draw (-8.4,-0.25) to (-8.4,-0.6) to (A3) node {$\cic{3}$} to (-3.1,-0.6) to (-3.1,-0.25);
\draw (-9.6,0.25) to (-9.6,0.6) to (A4) node {$\cic{4}$} to (2.8,0.6) to (2.8,0.25);
\draw (-8.5,-0.25) to (-8.5,-0.8) to (A5) node {$\cic{5}$} to (1.7,-0.8) to (1.7,-0.25);
%\draw (-1.2,0.25) to (-1.2,0.8) to (A8) node {$\cic{8}$} to (3.0,0.8) to (3.0,0.25);
\end{tikzpicture}

\begin{tikzpicture}[scale=0.9]
\node (A0) at (0.05,0) {$\dots,\eta_g,0,\dots,0,\eta_{g+1},\dots,\eta_h,0,\dots,0,\eta_{h+1},\dots;1$};
\node (A6) at (-0.8,0.4) {${}$};
\node (A7) at (-1,-0.5) {${}$};
\draw (-3.6,0.25) to (-3.6,0.4) to (A6) node {$\cic{6}$} to (1.8,0.4) to (1.8,0.25);
\draw (-2.5,-0.25) to (-2.5,-0.5) to (A7) node {$\cic{7}$} to (0.6,-0.5) to (0.6,-0.25);
\node (A) at (0,-1.4) {Possible cuts of $D_r t^\fm(\bfk)$.};
\end{tikzpicture}
\end{center}

We now show by induction on the depth of $\bfk$ that $D_r t^\fm(\bfk)$ is unramified. Let $\bfq=(q_1,\dots,q_d)$ and
$\bfk=(\{2n\}_{e},1,\bfq)=(k_1,\dots,k_\ell)$ with $\ell=e+1+d$. Put $\eta_{\ell+1}=1$ and
$a_m=k_m-1$ for all $1\le m\le \ell$. Then for all $e+2\le g\le \ell$ by substitution
$\eta_m\to \eta_m\eta_g$ for $1\le m<g$ we get
\begin{align*}
\ncic{1}=&\, \gd_{|\bfk_{1,g-1}|=r}\sum_{\eta_1,\dots,\eta_\ell=\pm1}\eta_1 I^\fl(0;\eta_1,\dots;\eta_g)\ot I^\fm(0;\eta_g,\dots;1) \\
=&\, \gd_{|\bfk_{1,g-1}|=r} \sum_{\eta_1,\dots,\eta_\ell=\pm1}\eta_1 \eta_g I^\fl(0;\eta_1,\dots;1)\ot I^\fm(0;\eta_g,\dots;1) \\
=&\,\gd_{|\bfk_{1,g-1}|=r} \, t^\fl(\bfk_{1,g-1}) \ot t^\fm(\bfk_{g,\ell}) .
\end{align*}
By induction, $t^\fl(\bfk_{1,g-1})$ is unramified (if $g=e+2$ then this follows from Prop.~\ref{prop:unrmMtVmodProd}). By Lemma~\ref{lem:MtVall>1}, $t^\fm(\bfk_{g,\ell})$ is unramified.
Next, for any fixed $1\le i<j\le e$
\begin{align*}
\ncic{2}=&\, \sum_{b=0}^{a_j-1} \gd_{|\bfk_{i,j-1}|+b=r} \sum_{\eta_1,\dots,\eta_\ell=\pm1}\eta_1 I^\fl(\eta_i;0_{a_i},\dots,\eta_j,0_b;0)\ot
I^\fm(0;\eta_1,\dots,\eta_i,0_{a_j-b},\eta_{j+1},\dots;1), \\
\ncic{3}=&\, \sum_{b=0}^{a_i-1} \gd_{|\bfk_{i+1,j}|+b=r} \sum_{\eta_1,\dots,\eta_\ell=\pm1}\eta_1 I^\fl(0;0_b,\eta_{i+1},\dots;\eta_{j+1})\ot
I^\fm(0;\eta_1,\dots,\eta_i,0_{a_i-b},\eta_{j+1},\dots;1)\\
=&\, -\sum_{b=0}^{a_i-1} \gd_{|\bfk_{i+1,j}|+b=r} \sum_{\eta_1,\dots,\eta_\ell=\pm1}\eta_1 I^\fl(\eta_{j+1};\dots,\eta_{i+1},0_b;0)\ot
I^\fm(0;\eta_1,\dots,\eta_i,0_{a_i-b},\eta_{j+1},\dots;1).
\end{align*}
Since $a_i=a_j$ we get $\ncic{2}+\ncic{3}=0$. Next, for any fixed $1\le i\le e+1, e+2\le g\le \ell$
\begin{align*}
\ncic{4}=&\, \sum_{b=0}^{a_g-1} \gd_{|\bfk_{i,g-1}|+b=r} \sum_{\eta_1,\dots,\eta_\ell=\pm1}\eta_1 I^\fl(\eta_i;0_{a_i},\dots,\eta_g,0_b;0)\ot
I^\fm(0;\eta_1,\dots,\eta_i,0_{a_g-b},\eta_{g+1},\dots;1) \\
=&\, - \sum_{b=0}^{a_g-1} \gd_{|\bfk_{i,g-1}|+b=r}\, \tz_b^\fl(\bfk_{g-1,i})\ot t^\fm(\bfk_{1,i-1},k_g-b,\bfk_{g+1,\ell})
\end{align*}
which is unramified by Thm.~\ref{thm-distribution} (if $i=e+1$ and $g>e+2$ then by Prop.~\ref{prop:unrmMtVmodProd})
and Lemma~\ref{lem:MtVall>1}, except for the case $i=e+1=g-1$ and $b=0$ when $\tz_b^\fl(\bfk_{g-1,i})=\tz^\fl(1)$ which is ramified. Similarly,
\begin{align*}
\ncic{5}=&\, \sum_{b=0}^{a_i-1} \gd_{|\bfk_{i+1,g-1}|+b=r} \sum_{\eta_1,\dots,\eta_\ell=\pm1}\eta_1 I^\fl(0;0_b,\eta_{i+1},\dots;\eta_{g})\ot
I^\fm(0;\eta_1,\dots,\eta_i,0_{a_i-b},\eta_{g},\dots;1)\\
=&\, \sum_{b=0}^{a_i-1} \gd_{|\bfk_{i+1,g-1}|+b=r} \, \tz_b^\fl(\bfk_{i+1,g-1})\ot t^\fm(\bfk_{1,i-1},k_i-b,\bfk_{g,\ell})
\end{align*}
which is unramified by Thm.~\ref{thm-distribution} (if $g=e+2$ and $i<e$ then by Prop.~\ref{prop:unrmMtVmodProd})
and Lemma~\ref{lem:MtVall>1}, except for the case $i=e=g-2$ and $b=0$ when $\tz_b^\fl(\bfk_{i+1,g-1})=\tz^\fl(1)$ which is ramified.
But this ramified term is canceled exactly by the exceptional ramified term of $\ncic{4}$\,.

Further, if $\dep(\bfq)=1$ then neither $\ncic{6}$\, and $\ncic{7}$\, can appear and therefore we have proved the base case of the induction.
If $\dep(\bfq)>1$ then both $\ncic{6}$\, and $\ncic{7}$\, have the shape of
$\tz_b^\fl(\bfs)\ot t^\fm(\{2n\}_{e},1,\bfp)$ where $b\ge0$, $\dep(\bfp)<\dep(\bfq)$, and every component of $\bfs$ and $\bfp$ is at least two.
By induction and Thm.~\ref{thm-distribution} we see that these are all unramified.

We have now completed the proof of the theorem by Thm.~\ref{thm-Glanois}.
\end{proof}

\begin{lem}\label{lem:MtVdifference}
Suppose all components of $\bfq$ are at least two. For any positive odd number $O>1$ and even number $E$
the difference $A=t^\fm(E,O,1,\bfq)-t^\fm(O,E,1,\bfq)$ is unramified.
\end{lem}
\begin{proof}
Clearly $D_1A=0$ by Lemma~\ref{lem:Charlton}. Assume $r>1$ is odd. For simplicity we set $\ell=3+\dep(\bfq)$,
$\bfk=(E,O,1,\bfq)$, $a=E-1$ and $b=O-1$.
Using Lemma~\ref{Murakami-lemma9} we only need to consider the following cuts for $D_r t^\fm(E,O,1,\bfq)$ either starting
at or ending at 0. Note that if a cut starts after $\eta_3$ (so $\dep(\bfq)>1$) then it is unramified by induction and Lemma~\ref{lem:MtVall>1}.

\begin{center}
\begin{tikzpicture}[scale=0.9]
\node (A0) at (0.05,0) {$0;\eta_1,0_{a-c-1},0,0_c,\eta_2,0_c,0,0_{b-c-1},\eta_3,\eta_4,\dots,\eta_i,\dots,0,\dots,\eta_{i+1},\dots;1$};
\node (A1) at (-2.6,-0.4) {${}$};
\node (A2) at (-4.2,-0.9) {${}$};
\node (A3) at (-3.6,0.4) {${}$};
\node (A4) at (-1.3,-0.6) {${}$};
\node (A5) at (-2.5,0.6) {${}$};
\node (A6) at (-1,-0.8) {${}$};
\node (A7) at (0.6,0.5) {${}$};
\node (A8) at (0.2,0.7) {${}$};
\node (A9) at (1.8,-0.5) {${}$};
\draw (-6.46,-0.25) to (-6.46,-0.4) to (A1) node {$\cic{1}$} to (.13,-0.4) to (.13,-0.25);
\draw (-6.5,-0.25) to (-6.5,-0.9) to (A2) node {$\cic{2}$} to (2.2,-0.9) to (2.2,-0.25);
\draw (-5.9,0.25) to (-5.9,0.4) to (A3) node {$\cic{3}$} to (-1.77,0.4) to (-1.77,0.25);
\draw (-5.9,-0.25) to (-5.9,-0.6) to (A4) node {$\cic{4}$} to (3.58,-0.6) to (3.58,-0.25);
\draw (-3.98,0.25) to (-3.98,0.6) to (A5) node {$\cic{5}$} to (.15,0.6) to (.15,0.25);
\draw (-3.98,-0.25) to (-3.98,-0.8) to (A6) node {$\cic{6}$} to (2.16,-0.8) to (2.16,-0.25);
\draw (-2.9,0.25) to (-2.9,0.5) to (A7) node {$\cic{7}$} to (3.56,0.5) to (3.56,0.25);
\draw (-1.73,0.25) to (-1.73,0.7) to (A8) node {$\cic{8}$} to (2.15,0.7) to (2.15,0.25);
\draw (.17,-0.25) to (.17,-0.5) to (A9) node {$\cic{9}$} to (3.54,-0.5) to (3.54,-0.25);
\node (A) at (0,-1.4) {Possible cuts of $D_r t^\fm(E,O,1,\bfq)$.};
\end{tikzpicture}
\end{center}

We have
\begin{align*}
\ncic{1}=&\, \gd_{E+O=r}\sum_{\eta_1,\dots,\eta_\ell=\pm1}\eta_1 I^\fl(0;\eta_1,0_a,\eta_2,0_b;\eta_3)\ot I^\fm(0;\eta_3,\eta_4,\dots;1) \\
=&\, \gd_{E+O=r} \, t^\fl(E,O) \ot t^\fm(1,\bfq)
= -\gd_{E+O=r} \, t^\fl(E+O) \ot t^\fm(1,\bfq) \quad\text{by Cor.~\ref{cor:DBreduce};}\\
\ncic{2}=&\, \sum_{i=4}^\ell \gd_{|\bfk_{1,i-1}|=r}\sum_{\eta_1,\dots,\eta_\ell=\pm1}\eta_1 I^\fl(0;\eta_1,\dots;\eta_i)\ot I^\fm(0;\eta_i,\dots;1) \\
=&\,\sum_{i=4}^\ell \gd_{|\bfk_{1,i-1}|=r} \, t^\fl(\bfk_{1,i-1}) \ot t^\fm(\bfk_{i,\ell}) \\
&\, \text{unramified by induction (or Prop.~\ref{prop:unrmMtVmodProd} if $i=4$) and Lemma~\ref{lem:MtVall>1};} \\
\ncic{3}=&\, \sum_{c=0}^{O-2} \gd_{E+c=r}\sum_{\eta_1,\dots,\eta_\ell=\pm1}\eta_1
 I^\fl(\eta_1;0_a,\eta_2,0_c;0)\ot I^\fm(0;\eta_1,0_{b-c},\eta_3,\eta_4,\dots;1) \\
=&\,- \sum_{c=0}^{O-2} \gd_{E+c=r} \, \tz_c^\fl(E) \ot t^\fm(O-c,1,\bfq) \\
&\, \text{unramified by Thm.~\ref{thm-distribution} and Thm.~\ref{thm:unramfiedMtV1} since $c$ is odd;} \\
\ncic{4}=&\, \sum_{i=4}^\ell \sum_{c=0}^{k_i-2} \gd_{|\bfk_{1,i-1}|+c=r}\sum_{\eta_1,\dots,\eta_\ell=\pm1}\eta_1
 I^\fl(\eta_1;\dots,\eta_i,0_c;0)\ot I^\fm(0;\eta_1,0_{k_i-c-1},\eta_{i+1},\dots;1) \\
=&\, - \sum_{i=4}^\ell \gd_{|\bfk_{1,i-1}|+c=r} \, \tz_c^\fl(\bfk_{i-1,1}) \ot t^\fm(k_i-c,\bfk_{i+1,\ell}) \\
&\, \text{unramified by Thm.~\ref{thm-distribution} (or Prop.~\ref{prop:unrmMtVmodProd} if $i=4$) and Lemma~\ref{lem:MtVall>1};}\\
\ncic{5}=&\, \sum_{c=0}^{E-2} \gd_{c+O=r}\sum_{\eta_1,\dots,\eta_\ell=\pm1}\eta_1
 I^\fl(0;0_c,\eta_2,0_b;\eta_3)\ot I^\fm(0;\eta_1,0_{a-c},\eta_3,\eta_4,\dots;1) \\
=&\, \sum_{c=0}^{E-2} \gd_{c+O=r} \, \tz_c^\fl(O) \ot t^\fm(E-c,1,\bfq) \\
&\, \text{unramified by Thm.~\ref{thm-distribution} and Thm.~\ref{thm:unramfiedMtV1} since $c$ is even;}\\
\ncic{6}=&\, \sum_{i=4}^\ell \sum_{c=0}^{E-2} \gd_{c+|\bfk_{2,i-1}|=r} \sum_{\eta_1,\dots,\eta_\ell=\pm1}\eta_1
 I^\fl(0;0_c,\eta_2,0_b,\eta_3,\dots;\eta_i)\ot I^\fm(0;\eta_1,0_{a-c},\eta_i,\dots;1) \\
=&\, \sum_{i=4}^\ell \sum_{c=0}^{E-2}\gd_{c+|\bfk_{2,i-1}|=r}\, \tz_c^\fl(\bfk_{2,i-1}) \ot t^\fm(E-c,\bfk_{i,\ell})\\
&\, \text{unramified by Thm.~\ref{thm-distribution} (or Prop.~\ref{prop:unrmMtVmodProd} if $i=4$) and Lemma~\ref{lem:MtVall>1};}\\
\ncic{7}=&\, \sum_{i=4}^\ell \sum_{c=0}^{k_i-2} \gd_{|\bfk_{2,i-1}|+c=r}\sum_{\eta_1,\dots,\eta_\ell=\pm1}\eta_1
 I^\fl(\eta_2;0_b,\eta_3,\dots,\eta_i,0_c;0)\ot I^\fm(0;\eta_1,0_a,\eta_2,0_{k_i-c-1},\eta_{i+1},\dots;1) \\
=&\, - \sum_{i=4}^\ell \sum_{c=0}^{k_i-2} \gd_{|\bfk_{2,i-1}|+c=r} \, \tz_c^\fl(\bfk_{i-1,2}) \ot t^\fm(E,k_i-c,\bfk_{i+1,\ell}) \\
&\, \text{unramified by Thm.~\ref{thm-distribution} and Lemma~\ref{lem:MtVall>1};}\\
\ncic{8}=&\, \sum_{i=4}^\ell \sum_{c=0}^{O-2} \gd_{|\bfk_{3,i-1}|+c=r}\sum_{\eta_1,\dots,\eta_\ell=\pm1}\eta_1
 I^\fl(0;0_c,\eta_3,\dots;\eta_i)\ot I^\fm(0;\eta_1,0_a,\eta_2,0_{b-c},\eta_i,\dots;1) \\
=&\, \sum_{i=4}^\ell \sum_{c=0}^{O-2}\gd_{|\bfk_{3,i-1}|+c=r} \, \tz_c^\fl(\bfk_{3,i-1}) \ot t^\fm(E,O-c,\bfk_{i,\ell}) \\
&\, \text{unramified by Thm.~\ref{thm-distribution} and Thm.~\ref{thm:unramfiedMtV1} (note that $r>1$ so if $c=0$ then $i>4$);}\\
\ncic{9}=&\, \sum_{i=4}^\ell \sum_{c=0}^{k_i-2} \gd_{|\bfk_{4,i-1}|+c+1=r}\sum_{\eta_1,\dots,\eta_\ell=\pm1}\eta_1
 I^\fl(\eta_3;\eta_4,\dots,\eta_i,0_c;0)\ot I^\fm(0;\eta_1,0_a,\eta_2,0_b,\eta_3,0_{k_i-c-1},\eta_{i+1},\dots;1) \\
=&\, - \sum_{i=4}^\ell \sum_{c=0}^{k_i-2} \gd_{|\bfk_{4,i-1}|+c+1=r} \, \tz_c^\fl(\bfk_{i-1,4},1) \ot t^\fm(E,O,k_i-c,\bfk_{i+1,\ell})
\end{align*}
which is unramified by Prop.~\ref{prop:unrmMtVmodProd} (if $i>4$) and Lemma~\ref{lem:MtVall>1}.
If $i=4$ then $c>0$ so $\tz_c^\fl(\bfk_{i-1,4},1)=\tz_c^\fl(1)$ is unramified.

The computation for $D_r t^\fm(O,E,1,\bfq)$ is completely similar such that all but $\ncic{1}$\, are unramified. But
the contribution from $\ncic{1}$\, is the same for $D_r t^\fm(O,E,1,\bfq)$ and $D_r t^\fm(E,O,1,\bfq)$
by Cor.~\ref{cor:DBreduce} and therefore $D_r A$ is always unramified.
This completes the proof of the lemma by Thm.~\ref{thm-Glanois}.
\end{proof}

\begin{thm}\label{thm:unramfiedMtV2}
Suppose $m,n\in\N$ and all components of $\bfq$ are at least $2$. Then the motivic MtV $t^\fm(2n,2m,2n,1,\bfq)$ is unramified.
\end{thm}

\begin{proof} Let $\bfk=(2n,2m,2n,1,\bfq)$ and $\ell=4+\dep(\bfq)$.
First, $D_1 t^\fm(\bfk)=0$ by Lemma~\ref{lem:Charlton}. Suppose $r>1$ is odd. Similar to the argument in Lemma~\ref{lem:MtVdifference}
we only need to consider the following four types of cuts for $D_r t^\fm(\bfk)$.

\begin{center}
\begin{tikzpicture}[scale=0.9]
\node (A0) at (0.05,0) {$0;\eta_1,0_{2n-c-2},0,0_c,\eta_2,0_{2m-1},\eta_3,0_c,0,0_{2n-c-2},\eta_4,\eta_5,\dots;1$};
\node (A1) at (-3.2,-0.4) {${}$};
\node (A2) at (1.43,-0.4) {${}$};
\node (A3) at (-1.3,0.4) {${}$};
\node (A4) at (-.1,0.6) {${}$};
\draw (-5.0,-0.25) to (-5.0,-0.4) to (A1) node {$\cic{1}$} to (-.8,-0.4) to (-.8,-0.25);
\draw (-.6,-0.25) to (-.6,-0.4) to (A2) node {$\cic{2}$} to (3.56,-0.4) to (3.56,-0.25);
\draw (-2.8,0.25) to (-2.8,0.4) to (A3) node {$\cic{3}$} to (0.2,0.4) to (0.2,0.25);
\draw (-1.6,0.25) to (-1.6,0.6) to (A4) node {$\cic{4}$} to (1.5,0.6) to (1.5,0.25);
\node (A) at (0,-1.4) {Possible cuts of $D_r t^\fm(2n,2m,2n,1,\bfq)$.};
\end{tikzpicture}
\end{center}

We have, for $0\le c\le 2m-2$
\begin{align*}
\ncic{1}=&\, \gd_{2n+c=r}\sum_{\eta_1,\dots,\eta_\ell=\pm1}\eta_1
 I^\fl(\eta_1;0_{2n-1},\eta_2,0_c;0)\ot I^\fm(0;\eta_1,0_{2m-c-1},0_{2n-1},\eta_4,\eta_5,\dots;1) \\
=&\, -\gd_{2n+c=r} \, \tz_c^\fl(2n) \ot t^\fm(2m-c,2n,1,\bfq), \\
\ncic{2}=&\, \gd_{2n+c=r}\sum_{\eta_1,\dots,\eta_\ell=\pm1}\eta_1
 I^\fl(0;0_c,\eta_3,0_{2n-1};\eta_4)\ot I^\fm(0;\eta_1,0_{2n-1},\eta_2,0_{2m-c-1},\eta_4,\eta_5,\dots;1) \\
=&\, \gd_{2n+c=r} \, \tz_c^\fl(2n) \ot t^\fm(2n,2m-c,1,\bfq),
\end{align*}
while for $0\le c\le 2n-2$
\begin{align*}
\ncic{3}=&\, \gd_{2m+c=r}\sum_{\eta_1,\dots,\eta_\ell=\pm1}\eta_1
 I^\fl(0;0_c,\eta_2,0_{2m-1};\eta_3)\ot I^\fm(0;\eta_1,0_{2n-c-1},\eta_3,0_{2n-1},\eta_4,\eta_5,\dots;1) \\
=&\, \gd_{2m+c=r} \, \tz_c^\fl(2m) \ot t^\fm(2n-c,2n,1,\bfq), \\
\ncic{4}=&\, \gd_{2m+c=r}\sum_{\eta_1,\dots,\eta_\ell=\pm1}\eta_1
 I^\fl(\eta_2;0_{2m-1},\eta_3,0_c;0)\ot I^\fm(0;\eta_1,0_{2n-1},\eta_2,0_{2n-c-1},\eta_4,\eta_5,\dots;1) \\
=&\, - \gd_{2m+c=r} \, \tz_c^\fl(2m) \ot t^\fm(2n,2n-c,1,\bfq).
\end{align*}
Hence, by Lemma~\ref{lem:MtVdifference} we see immediately that $\ncic{1}+\ncic{2}$\, and $\ncic{3}+\ncic{4}$\, are both unramified.
This completes the proof of the theorem by Thm.~\ref{thm-Glanois}.
\end{proof}

\begin{lem}\label{lem:MtV3to1}
Suppose all components of $\bfq$ are at least $2$. Then
\begin{equation*}
3t^\fm(3,2,1,\bfq)+t^\fm(2,1,2,1,\bfq)
\end{equation*}
is unramified.
\end{lem}

\begin{proof} Let $\bfk=(3,2,1,\bfq)$. Then clearly $D_1t^\fm(\bfk)=0$ by Lemma~\ref{lem:Charlton}.
Assume now $r>1$ is odd. By Lemma~\ref{Murakami-lemma9} we only need to consider the following cuts of $D_r t^\fm(\bfk)$:

\begin{center}
\begin{tikzpicture}[scale=0.9]
\node (A0) at (0.05,0) {$\hskip-1.4cm 0;\eta_1,0,0,\eta_2,0,\eta_3,\eta_4,\dots,\eta_i,\dots,0,\dots,\eta_{i+1},\dots;1$};
\node (A1) at (-4.4,-0.4) {${}$};
\node (A2) at (-4.0,0.4) {${}$};
\node (A3) at (-3.4,0.6) {${}$};
\node (A4) at (-2.9,-0.6) {${}$};
\node (A5) at (-1.2,-0.4) {${}$};
\node (A6) at (-1.7,0.4) {${}$};
\node (A7) at (-.7,0.6) {${}$};
\node (A8) at (-4.2,-0.7) {${}$};
\node (A9) at (-1.6,-0.8) {${}$};
\node (A10) at (-2.0,0.8) {${}$};
\node (A11) at (-2.4,1) {${}$};
\node (A12) at (-2.2,-1) {${}$};
\draw (-5.42,-0.25) to (-5.42,-0.4) to (A1) node {$\cic{1}$} to (-3.38,-0.4) to (-3.38,-0.25);
\draw (-4.96,0.25) to (-4.96,0.4) to (A2) node {$\cic{2}$} to (-3.02,0.4) to (-3.02,0.25);
\draw (-4.38,0.25) to (-4.38,0.6) to (A3) node {$\cic{3}$} to (-2.4,0.6) to (-2.4,0.25);
\draw (-3.96,-0.25) to (-3.96,-0.6) to (A4) node {$\cic{4}$} to (-1.8,-0.6) to (-1.8,-0.25);
\draw (-3.42,-0.25) to (-3.42,-0.4) to (A5) node {$\cic{5}$} to (1.0,-0.4) to (1.0,-0.25);
\draw (-2.98,0.25) to (-2.98,0.4) to (A6) node {$\cic{6}$} to (-0.42,.4) to (-0.42,0.25);
\draw (-2.34,0.25) to (-2.34,0.6) to (A7) node {$\cic{7}$} to (0.96,0.6) to (0.96,0.25);
\draw (-5.46,-0.25) to (-5.46,-0.7) to (A8) node {$\cic{8}$} to (-2.4,-0.7) to (-2.4,-0.25);
\draw (-5.5,-0.25) to (-5.5,-0.8) to (A9) node {$\cic{9}$} to (-0.42,-0.8) to (-0.42,-0.25);
\draw (-5.0,0.25) to (-5.0,0.8) to (A10) node {$\cic{\hskip-1.2pt1\!0}$} to (1.0,0.8) to (1.0,0.25);
\draw (-4.42,0.25) to (-4.42,1) to (A11) node {$\cic{\hskip-1.2pt1\!1}$} to (-0.38,1) to (-0.38,0.25);
\draw (-4.0,-0.25) to (-4.0,-1) to (A12) node {$\cic{\hskip-1.2pt1\!2}$} to (-0.38,-1) to (-0.38,-0.25);
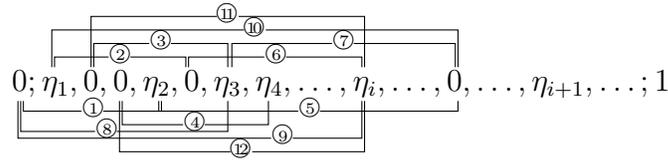
\node (A) at (0,-1.6) {Possible cuts of $D_r t^\fm(3,2,1,\bfq)$.};
\end{tikzpicture}
\end{center}

By tedious but straightforward computation we can see that
\begin{align*}
\ncic{1}=&\, \gd_{r=3}\, t^\fl(3) \ot t^\fm(2,1,\bfq), \quad
\ncic{2}=-\gd_{r=3} \, \tz^\fl(3) \ot t^\fm(2,1,\bfq), \\
\ncic{3}=&\, \gd_{r=3}\, \tz_1^\fl(2) \ot t^\fm(2,1,\bfq), \quad
\ncic{4}= -\gd_{r=3}\, \tz^\fl(2,1) \ot t^\fm(3,\bfq), \\
\ncic{5}=&\, -\sum_{i=4}^\ell \sum_{c=0}^{k_i-2} \gd_{|\bfk_{i-1,2}|+c=r}\, \tz_c^\fl(\bfk_{i-1,2}) \ot t^\fm(3,k_i-c,\bfk_{i+1,\ell}), \\
\ncic{6}=&\, \sum_{i=5}^\ell \gd_{|\bfk_{3,i-1}|=r}\, \tz^\fl(\bfk_{3,i-1}) \ot t^\fm(3,2,\bfk_{i,\ell}), \\
\ncic{7}=&\, -\sum_{i=4}^\ell \sum_{c=0}^{k_i-2} \gd_{|\bfk_{3,i-1}|+c=r}\, \tz_c^\fl(\bfk_{i-1,3}) \ot t^\fm(3,2,k_i-c,\bfk_{i+1,\ell}), \\
\ncic{8}=&\,  \gd_{r=5}\, t^\fl(3,2) \ot t^\fm(\bfk_{3,\ell}), \\
\ncic{9}=&\, \sum_{i=5}^\ell \gd_{|\bfk_{1,i-1}|=r}\, t^\fl(\bfk_{1,i-1}) \ot t^\fm(\bfk_{i,\ell}), \\
\ncic{\hskip-1.2pt1\!0}=&\, -\sum_{i=4}^\ell \sum_{c=0}^{k_i-2} \gd_{|\bfk_{1,i-1}|+c=r} \, \tz_c^\fl(\bfk_{i-1,1}) \ot t^\fm(k_i-c,\bfk_{i+1,\ell}), \\
\ncic{\hskip-1.2pt1\!1}=&\, \sum_{i=5}^\ell \gd_{|\bfk_{2,i-1}|+1=r}\, \tz_1^\fl(\bfk_{2,i-1}) \ot t^\fm(2,\bfk_{i,\ell}), \\
\ncic{\hskip-1.2pt1\!2}=&\, \sum_{i=5}^\ell \gd_{|\bfk_{2,i-1}|=r} \, \tz^\fl(\bfk_{2,i-1}) \ot t^\fm(3,\bfk_{i,\ell}).
\end{align*}
All terms $\tz_c^\fl(\bfl)\ot t^\fm(\bfs)$ above in $D_r t^\fm(3,2,1,\bfq)$ fall into three categories: either
$\bfs=(2,1,\bfq)$ or all components of $\bfs$ are at least two or $\bfs=(3,2,1,\bfq')$ with $\dep(\bfq')<\dep(\bfq)$.
Thus, all these $t^\fm(\bfs)$ are unramified by Thm.~\ref{thm:unramfiedMtV1}, Lemma~\ref{lem:MtVall>1} and induction, respectively.
The first factor $\tz_c(\bfl)$ in the above is always unramified by
Thm.~\ref{thm-distribution} or Prop.~\ref{prop:unrmMtVmodProd}.

On the other hand, if the terms have the shape $t^\fl(\bfl)\ot t^\fm(\bfs)$ then $\bfl=\bfk_{1,i}$ for some $i$ and $|\bfl|=r$ is odd,
then these terms must have the form (i) $t^\fl(3)\ot t^\fm(2,1,\bfq)$, or (ii) $t^\fl(3,2)\ot t^\fm(1,\bfq)$,
or (iii) $t^\fl(3,2,1,\bfq')\ot t^\fm(\bfq'')$ for some nontrivial decomposition $\bfq=(\bfq',\bfq'')$.
The first form is unramified by Thm.~\ref{thm:unramfiedMtV1}.

Similar computation is valid almost verbatim for $D_r t^\fm(2,1,2,1,\bfq)$. Thus we only need to consider the terms of
the forms (i)' $t^\fl(2,1)\ot t^\fm(2,1,\bfq)$, or (ii)' $t^\fl(2,1,2)\ot t^\fm(1,\bfq)$,
or (iii)' $t^\fl(2,1,2,1,\bfq')\ot t^\fm(\bfq'')$ for some nontrivial decomposition $\bfq=(\bfq',\bfq'')$.
The form (i)' is unramified by Cor.~\ref{cor:DBreduce} and Thm.~\ref{thm:unramfiedMtV1}.
For (ii)' and (iii)' we have to look at $3t^\fm(3,2,1,\bfq)+t^\fm(2,1,2,1,\bfq)$.
Then form (iii)+(iii)' gives rise to
$$\big(3t^\fl(3,2,1,\bfq')+t^\fl(2,1,2,1,\bfq')\big)\ot t^\fm(\bfq'')$$
which is unramified by induction.

Finally, we consider (ii)+(ii)' ((ii) appears as $\ncic{8}$ when $r=5$)
\begin{equation*}
 \big(3t^\fl(3,2)+t^\fl(2,1,2)\big)\ot t^\fm(1,\bfq).
\end{equation*}
But by easy computation (or the data mine \cite{BlumleinBrVe2010}, or \cite[Thm.~3.3]{Charlton2021} and the formula\footnote{Note that our normalization factors for MtVs are different from those used in \cite{Charlton2021}.} two lines above Remark 5.10 in \cite{Charlton2021}) we find
\begin{equation*}
t^\fm(3,2) =\frac47 t^\fm(2)t^\fm(3)- t^\fm(5) \quad\text{and}\quad t^\fm(2,1,2) =3t^\fm(5)-t^\fm(2)t^\fm(3).
\end{equation*}
Therefore,
\begin{equation*}
 (ii)+(ii)'=\big(3t^\fl(3,2)+t^\fl(2,1,2)\big)\ot t^\fm(1,\bfq)= \frac57 \big(t^\fm(2)t^\fm(3)\big)^\fl\ot t^\fm(1,\bfq) =0.
\end{equation*}
This completes the proof of the lemma.
\end{proof}

\begin{thm}\label{thm:unramfiedMtV3}
Suppose all components of $\bfq$ are at least $2$. Then the motivic MtV $t^\fm(2,1,3,2,1,\bfq)$ is unramified.
\end{thm}

\begin{proof} Let $\bfk=(2,1,3,2,1,\bfq)$ and $\ell=5+\dep(\bfq)$.
First, $D_1 t^\fm(\bfk)=0$ by Lemma~\ref{lem:Charlton}. We now assume $r>1$ is odd.
Using Lemma~\ref{Murakami-lemma9} we only need to consider the following cuts of $D_r t^\fm(2,1,3,2,1,\bfq)$:

\begin{center}
\begin{tikzpicture}[scale=0.9]
\node (A0) at (0.05,0) {$0;\eta_1,0,\eta_2,\eta_3,0,0,\eta_4,0,\eta_5,\eta_6,\dots,\eta_i,\dots,0,\dots,\eta_{i+1},\dots;1$};
\node (A1) at (-4.4,-0.4) {${}$};
\node (A2) at (-3.9,0.4) {${}$};
\node (A3) at (-2.3,0.6) {${}$};
\node (A4) at (-1.9,-0.4) {${}$};
\node (A5) at (-0.6,0.4) {${}$};
\node (A6) at (.3,0.7) {${}$};
\node (A7) at (0,-0.6) {${}$};
\node (A8) at (0.6,-0.4) {${}$};
\draw (-5.5,-0.25) to (-5.5,-0.4) to (A1) node {$\cic{1}$} to (-3.4,-0.4) to (-3.4,-0.25);
\draw (-5.0,0.25) to (-5.0,0.4) to (A2) node {$\cic{2}$} to (-2.8,0.4) to (-2.8,0.25);
\draw (-3.3,0.25) to (-3.3,0.6) to (A3) node {$\cic{3}$} to (-1.3,0.6) to (-1.3,0.25);
\draw (-2.8,-0.25) to (-2.8,-0.4) to (A4) node {$\cic{4}$} to (-.82,-0.4) to (-.82,-0.25);
\draw (-2.4,0.25) to (-2.4,0.4) to (A5) node {$\cic{5}$} to (1.3,.4) to (1.3,0.25);
\draw (-1.9,0.25) to (-1.9,0.7) to (A6) node {$\cic{6}$} to (2.6,0.7) to (2.6,0.25);
\draw (-1.3,-0.25) to (-1.3,-0.6) to (A7) node {$\cic{7}$} to (1.3,-0.6) to (1.3,-0.25);
\draw (-0.78,-0.25) to (-0.78,-0.4) to (A8) node {$\cic{8}$} to (2.6,-0.4) to (2.6,-0.25);
\end{tikzpicture}

\begin{tikzpicture}[scale=0.9]
\node (A0) at (0.05,0) {$0;\eta_1,0,\eta_2,\eta_3,0,0,\eta_4,0,\eta_5,\eta_6,\dots,\eta_i,\dots,0,\dots,\eta_{i+1},\dots;1$};
\node (A9) at (-2.1,-0.4) {${}$};
\node (A10) at (-1.2,0.4) {${}$};
\node (A11) at (-1.6,-0.6) {${}$};
\node (A12) at (-0.7,0.6) {${}$};
\node (A13) at (-0.4,-0.8) {${}$};
\node (A14) at (-0.9,0.8) {${}$};
\draw (-5.5,-0.25) to (-5.5,-0.4) to (A9) node {$\cic{9}$} to (1.26,-0.4) to (1.26,-0.25);
\draw (-5.0,0.25) to (-5.0,0.4) to (A10) node {$\cic{\hskip-1.2pt1\!0}$} to (2.56,0.4) to (2.56,0.25);
\draw (-4.5,-0.25) to (-4.5,-0.6) to (A11) node {$\cic{\hskip-1.2pt1\!1}$} to (1.3,-0.6) to (1.3,-0.25);
\draw (-3.9,0.25) to (-3.9,0.6) to (A12) node {$\cic{\hskip-1.2pt1\!2}$} to (2.6,0.6) to (2.6,0.25);
\draw (-3.4,-0.25) to (-3.4,-0.8) to (A13) node {$\cic{\hskip-1.2pt1\!3}$} to (2.6,-0.8) to (2.6,-0.25);
\draw (-2.8,0.25) to (-2.8,0.8) to (A14) node {$\cic{\hskip-1.2pt1\!4}$} to (1.3,0.8) to (1.3,0.25);
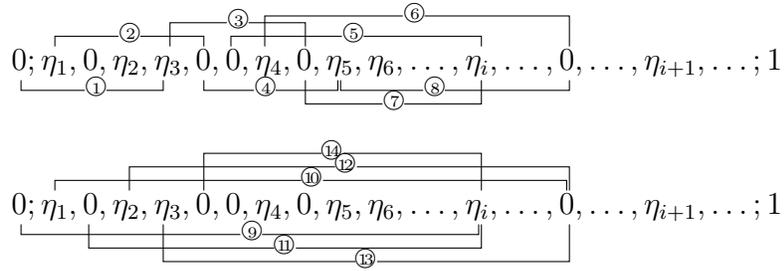
\node (A) at (0,-1.4) {Possible cuts of $D_r t^\fm(2,1,3,2,1,\bfq)$.};
\end{tikzpicture}
\end{center}

We have
\begin{align*}
\ncic{1}=&\, \gd_{r=3}\sum_{\eta_1,\dots,\eta_\ell=\pm1}\eta_1
 I^\fl(0;\eta_1,0,\eta_2;\eta_3)\ot I^\fm(0;\eta_3,0,0,\eta_4,0,\eta_5,\eta_6,\dots;1) \\
=&\, \gd_{r=3}\, t^\fl(2,1) \ot t^\fm(3,2,1,\bfq) = -\gd_{r=3} \, t^\fl(3) \ot t^\fm(3,2,1,\bfq) \quad(\text{by Cor.~\ref{cor:DBreduce}})\\
=&\, -\frac74 \gd_{r=3}\, \zeta^\fl(3) \ot t^\fm(3,2,1,\bfq)= -\gd_{r=3}\, 7\tz^\fl(3) \ot t^\fm(3,2,1,\bfq), \\
\ncic{2}=&\, \gd_{r=3}\, \sum_{\eta_1,\dots,\eta_\ell=\pm1}\eta_1
 I^\fl(\eta_1;0,\eta_2,\eta_3;0)\ot I^\fm(0;\eta_1,0,0,\eta_4,0,\eta_5,\eta_6,\dots;1) \\
=&\, - \gd_{r=3}\, \tz^\fl(1,2) \ot t^\fm(3,2,1,\bfq) \\
=&\, - \gd_{r=3}\, 2\tz^\fl(3) \ot t^\fm(3,2,1,\bfq) \quad\text{by Cor.~\ref{cor:DBreduce}}, \\
\ncic{3}=&\, \gd_{r=3}\sum_{\eta_1,\dots,\eta_\ell=\pm1}\eta_1
 I^\fl(\eta_3;0,0,\eta_4;0)\ot I^\fm(0;\eta_1,0,\eta_2,\eta_3,0,\eta_5,\eta_6,\dots;1) \\
=&\, - \gd_{r=3} \, \tz^\fl(3) \ot t^\fm(2,1,2,1,\bfq), \\
\ncic{4}=&\, \gd_{r=3}\sum_{\eta_1,\dots,\eta_\ell=\pm1}\eta_1
 I^\fl(0;0,\eta_4,0;\eta_5)\ot I^\fm(0;\eta_1,0,\eta_2,\eta_3,0,\eta_5,\eta_6,\dots;1) \\
=&\, \gd_{r=3} \, \tz_1^\fl(2) \ot t^\fm(2,1,2,1,\bfq) \\
=&\, -\gd_{r=3} \, 2\tz^\fl(3) \ot t^\fm(2,1,2,1,\bfq), \\
\ncic{5}=&\, \gd_{r=|\bfk_{4,i-1}|} \sum_{\eta_1,\dots,\eta_\ell=\pm1}\eta_1
 I^\fl(0;\eta_4,0,\eta_5,\eta_6,\dots;\eta_i)\ot I^\fm(0;\eta_1,0,\eta_2,\eta_3,0,0,\eta_i,\dots;1) \\
=&\, \gd_{r=|\bfk_{4,i-1}|} \, \tz^\fl(\bfk_{4,i-1}) \ot t^\fm(2,1,3,\bfk_{i,\ell}), \\
\ncic{6}=&\, \sum_{c=0}^{k_i-2}\gd_{r=|\bfk_{4,i-1}|+c} \sum_{\eta_1,\dots,\eta_\ell=\pm1}\eta_1
 I^\fl(\eta_4;0,\eta_5,\eta_6,\dots,\eta_i,0_c;0) \\
 &\, \hskip4cm \ot I^\fm(0;\eta_1,0,\eta_2,\eta_3,0,0,\eta_4,0_{k_i-c-1},\eta_{i+1},\dots;1) \\
=&\, -\gd_{r=|\bfk_{4,i-1}|+c} \, \tz_c^\fl(\bfk_{i-1,4}) \ot t^\fm(2,1,3,k_i-c,\bfk_{i+1,\ell}), \\
\ncic{7}=&\, \gd_{r=|\bfk_{5,i-1}|} \sum_{\eta_1,\dots,\eta_\ell=\pm1}\eta_1
 I^\fl(0;\eta_5,\eta_6,\dots;\eta_i)\ot I^\fm(0;\eta_1,0,\eta_2,\eta_3,0,0,\eta_4,0,\eta_i,\dots;1) \\
=&\, \gd_{r=|\bfk_{5,i-1}|} \, \tz^\fl(\bfk_{5,i-1}) \ot t^\fm(2,1,3,2,\bfk_{i,\ell}), \\
\ncic{8}=&\, \sum_{c=0}^{k_i-2}\gd_{r=|\bfk_{5,i-1}|+c} \sum_{\eta_1,\dots,\eta_\ell=\pm1}\eta_1
 I^\fl(\eta_5;\eta_6,\dots,\eta_i,0_c;0) \\
 &\, \hskip4cm \ot I^\fm(0;\eta_1,0,\eta_2,\eta_3,0,0,\eta_4,0,\eta_5,0_{k_i-c-1},\eta_{i+1},\dots;1) \\
=&\, -\gd_{r=|\bfk_{5,i-1}|+c} \, \tz_c^\fl(\bfk_{i-1,5}) \ot t^\fm(2,1,3,2,k_i-c,\bfk_{i+1,\ell}).
\end{align*}
Hence,
\begin{equation*}
\ncic{1}+\ncic{2}+\ncic{3}+\ncic{4} = -3\gd_{r=3}\, \tz^\fl(3) \ot \Big(3t^\fm(3,2,1,\bfq)+t^\fm(2,1,2,1,\bfq)\Big)
\end{equation*}
which is unramified by Lemma~\ref{lem:MtV3to1}.
Observe that $\ncic{5}$, $\ncic{6}$, $\ncic{7}$\, and $\ncic{8}$\, are all unramified by Thm.~\ref{thm-distribution} (by Prop.~\ref{prop:unrmMtVmodProd} if $i=6$ for $\ncic{5}$\, and $\ncic{6}$)
for the left factor and Thm.~\ref{thm:unramfiedMtV1} for the right factor. Note also that $i=6$ is impossible for $\ncic{7}$\,
while $i=6$ is allowed in  $\ncic{8}$\, with $\tz_c^\fl(\bfk_5)=\tz_c^\fl(1)=(-1)^c\tz^\fl(c+1)=\tz^\fl(r)$ unramified.
Similarly,
\begin{align*}
\ncic{9}=&\, \sum_{i=6}^\ell \gd_{|\bfk_{1,i-1}|=r} \, t^\fl(\bfk_{1,i-1}) \ot t^\fm(\bfk_{i,\ell})
\end{align*}
whose right factor is unramified by Lemma~\ref{Murakami-lemma9} and whose left factor is unramified by induction
(or by Prop.~\ref{prop:unrmMtVmodProd} if $i=6$).

All other cuts in $\ncic{\hskip-1.2pt1\!0}$-$\ncic{\hskip-1.2pt1\!4}$
have the form $\tz_c^\fl(\bfl)\ot t^\fm(\bfs)$ which fall into three categories: either
all components of $\bfs$ are at least two, or $\bfs=(2,1,\bfq')$ with all components of $\bfq'$ are at least two,
or $\bfs=(2,1,3,2,1,\bfq')$ with $\dep(\bfq')<\dep(\bfq)$.
Thus, all these $t^\fm(\bfs)$'s are unramified by Thm.~\ref{thm:unramfiedMtV1}, Lemma~\ref{lem:MtVall>1} and induction, respectively.
The first factor $\tz^\fl_c(\bfl)$ in the above is always unramified by
Thm.~\ref{thm-distribution} or Prop.~\ref{prop:unrmMtVmodProd}. Note that at most two unit components can appear in $\bfl$ and
they must be at the one or both ends if they do appear.

To summarize, we have shown that $D_1 t^\fm(\bfk)=0$ and $D_r t^\fm(\bfk)$ are unramified for all odd $r<\ell$.
This finishes the proof of the theorem by Thm.~\ref{thm-Glanois}.
\end{proof}

By Lemma~\ref{lem:Charlton},  $t^\fm(\bfk)$ is ramified if $t^\fm(\bfk)$ has a unit component at either or both ends.
We now provide a family of ramified motivic MtV $t^\fm(\bfk)$ with a unique unit component not appearing at either end.
Moreover, exactly one component to the left of this unit component is odd, namely, they have the form:
\begin{equation*}
    t(E_a,O,E_b,1,\dots)
\end{equation*}
where $a,b\ge0$, $E$'s are (possibly distinct) even numbers and $O$ is an odd number.

\begin{thm}\label{thm:ramfiedMtV1}
Let $1\le j< n<\ell$. Suppose $k_j\ge 3$ is odd and $k_i$ is even for all $1\le i<n$ and $i\ne j$.
Suppose further that $k_n=1$ and $k_i\ge 2$ for all $n< i\le \ell$.
Then the motivic MtV $t^\fm(k_1,\dots,k_\ell)$ is ramified.
\end{thm}
\begin{proof}
Let $\bfk=(k_1,\dots,k_\ell)$ and $r=|\bfk_{1,n-1}|$. Then $r$ is odd by the given conditions. We show that $D_r t^\fm(\bfk)$ is ramified.
Indeed, the cut starting at the first 0 and ends at $\eta_n$ yields
\begin{equation}\label{equ:ramMtV}
t^\fl(\bfk_{1,n-1})\ot t^\fm(1,\bfk_{n+1,\ell})
\end{equation}
which is ramified since $D_1 t^\fm(1,\bfk_{n+1,\ell})\ne 0$ by Lemma~\ref{lem:Charlton}. All other cuts of $D_rt^\fm(\bfk)$
have the form $\pm \tz_c^\fl(\bfl)\ot t^\fm(\bfs)$ with $\tz_c^\fl(\bfl)$ being unramified
and the components at the two ends of $\bfs$ both $\ge 2$. Thus $D_1t^\fm(\bfs)=0$ by Lemma~\ref{lem:Charlton}.
Consequently, the unramified term \eqref{equ:ramMtV} cannot be canceled. This completes the proof of the theorem.
\end{proof}

\section{Concluding remarks}
In this paper, we have applied Brown's motivic MZV theory and Glanois's descent theory of Euler sums, further developed by Murakami and Charlton, to study the ramified and unramified motivic mixed values which are variants of the multiple zeta values of level two. We are able to extend a result of Murakami on a conjecture of Kaneko and Tsumura concerning MTVs. Murakami showed that certain conditions discovered by Kaneko and Tsumura are sufficient for MTVs to be unramified. On the motivic level, we showed in Thm.~\ref{thm:MSV} that these conditions, at least for depth less than four, are also necessary.

We are also able to generalize a result of Charlton (see Thm.~\ref{thm:unramfiedMtV-MC}(2))
concerning a family of unramified MtVs with unit components
and provide two more such families in Thm.~\ref{thm:unramfiedMtV}.
We remark that Thm.~\ref{thm:unramfiedMtV}(2) is a special case of the following more general
conjecture for which we have some quite strong numerical evidence.

\begin{conj}\label{conj:unramfiedMtV}
Let $m,n,e\in\N$ and suppose that all components of $\bfq$ are at least two.
Then the MtV $t(\{2n\}_{e},2m,\{2n\}_{e},1,\bfq)$ is unramified.
\end{conj}
We further provide a family of ramified motivic MtVs in Thm.~\ref{thm:ramfiedMtV}.

Furthermore, we are able to describe completely when a motivic MSV is ramified if its depth is bounded by three in Thm.~\ref{thm:MSV}.
We also numerically searched for possible unramified MSVs of depth greater than three and found
\begin{equation*}
S(2,1,1,1,4)=\tfrac{3577}{768}\zeta(9)+\tfrac{343}{384}\zeta(3)^3-\tfrac{63}{64}\zeta(2)^3\zeta(3)-\tfrac{217}{640}\zeta(2)^2\zeta(5)
\end{equation*}
seems to be the only one when the weight is at most 15. Charlton kindly informed us that using the MZV data mine \cite{BlumleinBrVe2010} he checked that $S(2,1,1,1,4)$ is indeed the only unramified MSVs of depth greater than three if the weight is less than 12. We now end our paper with the following questions.

\begin{prob}\label{prob:unramfiedMSVandMtV}
Suppose the depth of $\bfk$ is at least four.
 \begin{enumerate}
\item [\upshape{(1)}] Is $S(\bfk)$ unramified only when $\bfk=(2,1,1,1,4)$?

\item [\upshape{(2)}] Is $T(\bfk)$ unramified for any $\bfk$ with depth greater than 4?

\item [\upshape{(3)}] Does Thm.~\ref{thm:unramfiedMtV}(1) and Conjecture~\ref{conj:unramfiedMtV}
exhaust all unramified MtV $t(\bfk)$ where $\bfk$ has exactly one unit component?

\item [\upshape{(4)}] Does Thm.~\ref{thm:unramfiedMtV}(3) exhaust all unramified MtV $t(\bfk)$ where $\bfk$ has exactly two unit components?

\item [\upshape{(5)}] Is there any unramified MtV $t(\bfk)$ where $\bfk$ has at least three unit components?
\end{enumerate}
\end{prob}

We want to point out that due the appearance of a high depth unramified multiple $S$-value $S(2,1,1,1,4)$, it might be very interesting to investigate high depth multiple $T$-value and see if Kaneko--Tsumura conjecture still holds.

\medskip
{\bf Acknowledgments.} Ce Xu is supported by the General Program of Natural Science Foundation of Anhui Province (Grant No. 2508085MA014). Jianqiang Zhao is supported by the Jacobs Prize from The Bishop's School. Both authors would like to thank Prof.~F.~Xu at the Capital Normal University and Prof.~C.~Bai at the Chern Institute of Mathematics for their warm hospitality, and Dr.~S.~Charlton for his helpful comments on an early version of the paper.

\bigskip
\noindent
\textbf{Statement of conflict of interest.} The authors claim that there is no conflict of interest.

\medskip
\noindent
\textbf{Data availability statement.} Our manuscript has no associated data.

\end{document}